\documentclass[12pt]{article}
\usepackage{graphicx}
\usepackage{amsmath}
\usepackage{amssymb}
\usepackage{theorem}
\usepackage{hyperref}
\usepackage{color}  

\numberwithin{equation}{section}

\newtheorem{thm}{Theorem}[section]
\newtheorem{lemma}[thm]{Lemma}

\newtheorem{prop}[thm]{Proposition}
\newtheorem{cor}[thm]{Corollary}
{\theorembodyfont{\rmfamily}
\newtheorem{defn}[thm]{Definition}
\newtheorem{rmk}[thm]{Remark}
}

\newcommand{\qed}{\hfill \mbox{\raggedright \rule{.07in}{.1in}}}

\newenvironment{proof}{\vspace{1ex}\noindent{\bf
Proof}\hspace{0.5em}}{\hfill\qed\vspace{1ex}}
\newenvironment{pfof}[1]{\vspace{1ex}\noindent{\bf Proof of #1}\hspace{0.5em}}{\hfill\qed\vspace{1ex}}

\newcommand{\B}{{\mathcal B}}
\newcommand{\R}{{\mathbb R}}

\newcommand{\E}{{\mathbb E}}
\newcommand{\N}{{\mathbb N}}

\newcommand{\M}{{\mathcal M}}
\renewcommand{\P}{{\mathbb P}}

\newcommand{\Leb}{\operatorname{Leb}}

\def\eps{\varepsilon}

\def\GR{G}

\def\s{\mathcal{S}}
\def\r{\sigma}
\def\uu{\chi}

\def\subjclass#1{\par\medskip
\noindent\textbf{Mathematics Subject Classification (2010):} #1}
\def\keywords#1{\par\medskip
\noindent\textbf{Keywords.} #1}

\begin{document}

\title{The pressure function for infinite equilibrium measures}
\author{Henk Bruin 
\thanks{Faculty of Mathematics, University of Vienna, 
Oskar Morgensternplatz 1, 1090 Vienna, Austria, {\it henk.bruin@univie.ac.at}}
\and Dalia Terhesiu 
\thanks{Department of Mathematics, University of Exeter, Streatham Campus,
North Park Road, Exeter EX4 4QF, UK, {\it daliaterhesiu@gmail.com}}
\and Mike Todd
\thanks{Mathematical Institute, University of St Andrews
North Haugh, St Andrews KY16 9SS, Scotland, {\it m.todd@st-andrews.ac.uk}}
}

\date{\today} 

\maketitle

\abstract{Assume that $(X,f)$ is a dynamical system and $\phi:X \to [-\infty, \infty)$ is a potential
such that the $f$-invariant measure $\mu_\phi$ equivalent to the $\phi$-conformal
measure is infinite, but that there is an inducing scheme $F = f^\tau$
with a finite measure $\mu_{\bar\phi}$ and polynomial tails $\mu_{\bar\phi}(\tau \geq n) = O(n^{-\beta})$,
$\beta \in (0,1)$.
We give conditions under which the pressure of $f$ for a perturbed
potential $\phi+s\psi$ relates to the pressure of the induced system as
$P(\phi+s\psi) = (C P(\overline{\phi+s\psi}))^{1/\beta} (1+o(1))$,
together with estimates for the $o(1)$-error term.
This extends results from Sarig~\cite{Sarig06} to the setting of infinite equilibrium states.
We give several examples of such systems, thus improving on the results of Lopes~\cite{L93}
for the Pomeau-Manneville map with potential $\phi_t = - t\log f'$, as well
as on the results by Bruin \& Todd ~\cite{BT07, BT12} on countably piecewise linear unimodal Fibonacci maps.
In addition, limit properties of the family of measures $\mu_{\phi+s\psi}$ as $s\to 0$ are 
studied and statistical properties (correlation coefficients and arcsine laws) 
under the limit measure are derived.}

\iffalse 
  	%One-parameter continuous families of measure-preserving transformations

\fi

\section{Introduction}\label{sec:intro}
 The (variational) pressure of a dynamical system $(X,f)$ and potential $\phi:X \to [-\infty, \infty)$
is
\begin{equation}\label{eq:Pvar}
P(\phi) := \sup\left\{ h_\mu(f) + \int_X \phi \, d\mu : \mu \in \M \text{ and } -\int \phi~d\mu<\infty\right\},
\end{equation}
where $\M$ is the class of $f$-invariant probability measures and $h_\mu(f)$ is the Kolmogorov entropy.
If a measure $\mu_\phi \in \M$ achieves this supremum, then it is called an {\em equilibrium measure}.
In ``good'' cases \cite{Bow75}, $\mu_\phi$ can be obtained as $d\mu_\phi = v \ dm_\phi$,
where $v$ is the normalised eigenfunction of the transfer operator
$L_\phi v(x) = \sum_{f(y) = x} e^{\phi(y)} v(y)$ associated to its leading eigenvalue
$\lambda(\phi)$, and $m_\phi$ (the conformal measure) is the eigenmeasure
of the dual operator $L^*_\phi$ associated to the same eigenvalue.
Moreover, this eigenvalue satisfies $\lambda(\phi) = e^{P(\phi)}$.

Frequently, one considers parametrised families of potentials $\phi_t = t\phi$,
where $T = 1/t$ is (classically) called ``temperature''.  If the (continuous) function
$t \mapsto P(\phi_t)$ fails to be real analytic at $t_0$, we speak of a {\em phase transition}. 
The largest $r \geq 1$ such that this function is $C^{r-1}$, but not $C^r$ at $t_0$ is called the {\em order}
of the phase transition.
In the examples we know of, an order $> 1$ indicates that as $t \to t_0$, $\mu_{\phi_t}$
does not converge to a finite equilibrium measure that is absolutely continuous
w.r.t.\ $t_0$-conformal measure $m_{t_0\phi}$.
However, it is possible that $m_{t_0\, \phi}$ exists, and a measure $\mu_{t_0 \, \phi}$
such that $d\mu_{t \, \phi_0} = v dm_{t_0 \, \phi}$ exists as well,
although the density $v \notin L^1(m_{t_0 \, \phi})$; this is the null recurrent case.
It is these ``infinite equilibrium states'' that are the topic of this paper.

A classical result of Fisher \& Felderhof \cite{FF70} presents physical systems
with phase transitions of arbitrary high order.
The standard mathematical example of this phenomenon is the Pomeau-Manneville
map $f:[0,1] \to [0, 1]$ (see \eqref{eq:nMPM}) with $\phi_t = -t \log f'$,  
where the order of tangency at the neutral fixed point
is the order of the phase transition: $P(\phi_t) \sim C (1-t)^\alpha$ for $t \lesssim 1$,
see Lopes \cite{L93}. His proof is based on the full shift with a potential mimicking $\phi_t$,
but the step from the symbolic to the (nonlinear) Manneville-Pomeau case is not entirely clear to us. 
A new proof is provided in this paper; for details, we refer to the statement and the proof of 
Proposition~\ref{prop:psilogtau}.

There are several reasons to be interested in the shape and smoothness of
the pressure function at phase transitions.
For example, in multifractal analysis,
the multifractal spectrum can often be realised as the Legendre transform of a pressure function.  
In good cases, this means that locally, it is the inverse of the derivative of the pressure function.  
So the smoothness of the pressure dictates the smoothness of the spectrum.  In particular, as above, 
in our examples we focus on the potentials $-t\log|f'|$, so our results have consequences for the Lyapunov spectrum.

The shape of the pressure function for the perturbed potential $\phi+s\psi$ ($s \approx 0$) also 
relates to limit laws of ergodic averages of $\psi$.
This is classical when $\mu_\phi$ is finite and $\frac{d}{ds} P(\phi+s\psi)$ and
$\frac{d^2}{ds^2}P(\phi+s\psi)$ are both finite: they are the expectation and variance, respectively, 
in the Central Limit Theorem that $\psi$ satisfies, see \cite[Propositions 4.10 and 4.11]{PP}.
Sarig \cite[Theorem 2]{Sarig06} shows that in the particular case that the system is Markov and satisfies the 
big image and preimage (BIP) property (see Section~\ref{sec:thermo}), then the following 
can be shown (where we scale $P(\phi) = 0$): $\psi$ is in the domain 
of a $\alpha$-stable law 
for some $\alpha \in (1,2)$  if and only if $P(\phi+s\psi) \sim C s + s^{\alpha} \ell(1/s)$ for some 
slowly varying function $\ell$. Further results \cite[Theorems 3--5]{Sarig06} deal with the cases  $\alpha \in (0,1]$ and
(non-Gaussian) $\alpha = 2$.  However, for Markov systems without BIP, the standard procedure 
to obtain a limit law is to pull this back from an induced, well-behaved, system 
\cite{MelTor04, Zwe07, Gouezel10}, and the same holds for the shape of the pressure function, \cite[Theorem 8]{Sarig06}.
For Markov systems without BIP, the connection between the asymptotic behaviour of pressure function 
and limit laws can be established once
each one of them is obtained via the pressure function/limit laws for the induced system 
as in~\cite[Section 5]{Sarig06}.

In this work we focus on the null recurrent potential case and
the behaviour of the pressure function 
can be related to the presence of limit laws by combining 
Theorem~\ref{thm:PbarP} with the results in Section~\ref{sec-correl}.

\subsection{Main results}

Here we give a rough outline of our results.
Suppose that $(X,f)$ is a non-uniformly hyperbolic system that allows a uniformly
hyperbolic induced map $(Y, F = f^\tau)$, and suppose that $\mu_\phi$ and $\mu_{\bar\phi}$ are
the equilibrium measures for $\phi$ and the induced potential $\bar\phi = \sum_{j=0}^{\tau-1} \phi \circ f^j$
respectively.
Sarig \cite[Theorem 8]{Sarig06}, in the case that $\mu_\phi$ is a probability measure and $P(\phi) = 0$,
gives a general relation between the pressure $P(\phi+s\psi)$ of the original system $f:X \to X$
and the pressure $P(\overline{\phi+s\psi})$ of the induced system $F:Y \to Y$: asymptotically (as $s\searrow 0$)
the two differ by a multiplicative constant $\mu_\phi(Y)$.

Our main result gives such a relation when  $\mu_\phi$ is infinite, and the induced potential $\bar\phi$ satisfies certain abstract conditions;
in particular we require a refined form of 
 $\mu_{\bar\phi}(\tau > n) = cn^{-\beta}(1+o(1))$, for $\beta \in (0,1)$. Our  abstract assumptions on $F$, $\bar\phi$ and $\bar\psi$
are formulated in Section~\ref{sec:abstsu}.  Under such assumptions, we show that there is $C > 0$ such that 
 \begin{equation}\label{eq:PbarP0}
  P(\phi + s\psi) = (CP(\overline{\phi + s\psi}))^{1/\beta} \cdot (1+o(1)) \quad \text{ as } \quad s \to 0.
 \end{equation}
The present Theorem~\ref{thm:PbarP} provides a refined form of~\eqref{eq:PbarP0} with precise error terms.
 In Section~\ref{sec:ex} we verify the set of abstract assumptions in Section~\ref{sec:abstsu}
for several examples of interest, as summarized below. In particular, in Section~\ref{sec:ex} 
we  show that potentials $\psi$ with induced version $\bar\psi$   satisfying the abstract assumptions in Section~\ref{sec:abstsu} include:
a) potentials $\psi \in L^1(\mu_\phi)$ (usually the geometric potential
 $\psi = \log|f'|$ falls in this class: see Remark~\ref{rmk-log}); b) potentials that are bounded above, but have heavy (negative) tails. 

 More generally, given the potentials $\phi+s\psi$, one is led to questions about the limit behaviour 
 of the measures $\mu_{\phi+s\psi}$ as $s \to 0$. Under a further abstract assumption on $\psi$ formulated in Section~\ref{sec-LP}, 
 we show that the induced measures
 $\mu_{\overline{\phi+s\psi}}$ tend to $\mu_{\bar\phi}$ in a sense stronger than the 
 weak${}^*$ topology; this is the content of  Lemma~\ref{lemma-contmeas}. Since $\mu_\phi$ is infinite, the expectations 
 $\E_{\overline{\phi+s\psi}}(\tau)$ of the inducing time $\tau$ tend to infinity as $s \to 0$ and
 in Lemma~\ref{lemma-tauexpec} we estimate the speed at which this happens.
 In the spirit of quasistatic dynamical systems (see e.g.\ \cite{DS15} and references therein), 
 we show in Theorem~\ref{thm-mix} that the correlation coefficients
 $\rho_{n,s}(v,w) := \int_Y v \ w \circ f^n \, d\mu_{\phi+s\psi}$ behave asymptotically the 
 same as in the case of  the null recurrent $\phi$,
 as $n \to \infty$, $s \to 0$ simultaneously, provided $s = o(n^{(\beta-1)/\beta -\eps})$.

 We give various examples of applications of our theory in Section~\ref{sec:ex}.  The main task there is to show that the abstract conditions formulated earlier in the paper hold.  
 Some of these examples are well-studied in infinite ergodic theory, so we can take much of the theory `off the shelf', but we also give some new classes of examples and develop the required theory here.
  
 An important class of maps with induced system with polynomial tails are
 so-called AFN-maps, i.e., non-uniformly expanding interval maps with neutral fixed points.
 The standard example here is the Pomeau-Manneville map (possibly non-Markov, see \eqref{eq:nMPM}),
 and we improve the result of Lopes \cite{L93} in this case:
 If the parameter $\alpha > 1$ (infinite measure case), then there exists
 $C > 0$ (made explicit in the proof of Proposition~\ref{prop:psilogtau}) such that
$$
P(\phi_t) = C(1+o(1))(1-t)^\alpha \qquad  \text{ as } t \nearrow 1.
$$
We also show in Proposition~\ref{prop:psilogtau2} that Theorem~\ref{thm:PbarP} applies to a class of 
unimodal Misiurewicz maps with flat critical points, i.e., the critical $c$ is not recurrent
and all the derivatives $D^nf$ vanish at $c$.

Another interesting set of examples are Fibonacci unimodal maps with piecewise linear branches,
which have not been greatly studied in this context.
These have `almost' first return maps without full branches (the first return maps themselves are studied in isolation in Proposition~\ref{prop:psilogtau3})
and a lacunary sequence of inducing times; hence there is no regular variation of tails and Theorem~\ref{thm:PbarP} 
does not apply.  Therefore we need to weaken the assumption $\mu_{\bar\phi}(\tau > n) = cn^{-\beta}(1+o(1))$, for $\beta \in (0,1)$ used in Theorem~\ref{thm:PbarP}.
Under an appropriate assumption, we prove Theorem~\ref{thm:PbarP2} which gives us a similar statement to 
 \eqref{eq:PbarP0}, but with $C$ replaced by some $C_1^*$ for an upper bound and $C_2^*$ for a lower bound. 
Once the tail estimates for these Fibonacci maps are proved in Theorem~\ref{thm:Fibo}, we obtain the asymptotics of the pressure function in Proposition~\ref{prop:fibpress}.
\\[3mm]
{\bf Acknowledgements:} The authors  would like to thank the Erwin Schr\"odinger Institute
where this paper was initiated during a ``Research in Teams'' project.  
MT and DT would like to thank the University of Exeter and the
 University of St Andrews, respectively, for their hospitality during mutual
short visits.

\section{Preliminaries}\label{sec:prelim}

{\bf Notation:}
We use notation $a_n \sim b_n$ if $a_n/b_n \to 1$, and
$a_n \asymp b_n$ if there is a constant $C > 1$ such that
$1/C \leq \liminf_n a_n/b_n \leq \limsup_n a_n/b_n \leq C$.
Also, we use ``big O'' and Vinogradov $\ll$ notation interchangeably, writing
$a_n=O(b_n)$ or $a_n\ll b_n$ as $n\to\infty$ if there is a constant
$C>0$ such that $a_n\le Cb_n$ for all $n\ge1$.  
Finally we write $a_n=o(b_n)$ if $\lim_n a_n/b_n=0$.

\subsection{Thermodynamic formalism for Markov maps}\label{sec:thermo}
Let $F: Y \to Y$ be a Markov map with countable Markov partition 
$\{ Y_i \}_{i \in \N}$.  That is, $Y=\cup_i Y_i$ and for each $i, j\in \N$, either  $X_i\cap F(X_j)=\emptyset$ or   $X_i\subset F(X_j)$.  Further, we assume that this satisfies the {\em big image and preimage} property (BIP), 
i.e., there is $N \in \N$ such that for all $i \in \N$ there are $a,b \in \{ 1, \dots, N\}$
such that $F(Y_a) \cap Y_i \neq \emptyset \neq F(Y_i) \cap Y_b$.
In particular, Markov systems with full branches ($F(Y_i)=Y$ for all $i$) have the BIP property.

Given $\phi:Y \to \R$, let $V_n(\phi) = \sup_{C_n} \sup_{x,y \in C_n} |\phi(x)-\phi(y)|$ be the 
\emph{$n$th variation}, 
where the first supremum is over all dynamically defined $n$-cylinders $C_n$, where this is of the form $C_n=[x_0, \ldots, x_{n-1}]=\{y\in Y: F^i(y)\in Y_{x_i} \text{ for } i=0, \ldots, n-1\}$.  
If $\sum_{n\ge 2} V_n(\phi)<\infty$ then we say that $\phi$ has \emph{summable variations} and write $\phi \in SV$.  
For such a potential,
define partition functions
\begin{equation}\label{eq:Zn}
Z_n(\phi, Y_i):=\sum_{x \in Y_i, F^nx = x} e^{S_n\phi(x)} 
\quad \text{ and } \quad
Z_n^*(\phi, Y_i) := \sum_{\stackrel{x \in Y_i, F^{n}x=x,}{F^kx \notin Y_i \ \mbox{\tiny for}\ 0< k < n}} 
\hspace{-10mm} e^{S_n\phi(x)},
\end{equation}
where $S_n\phi(x) = \phi(x) + \dots + \phi \circ F^{n-1}(x)$ is the $n$th ergodic sum.
The {\em Gurevich pressure}
\begin{equation}\label{eq:Pgur}
 P_G(\phi) := \lim_{n \to \infty} \frac{1}{n} \log Z_n(\phi, Y_i)
\end{equation}
exists, is independent of the state $Y_i$ (so we will drop $Y_i$ in the notation), 
and is equal to the variational pressure
from \eqref{eq:Pvar} whenever $|\sum_{Fy=x} e^{\phi(y)}|_\infty < \infty$ or BIP holds, 
see \cite[Theorem 3]{Sarig99}.

The potential $\phi$ is called {\em recurrent} if $\sum_n e^{-nP_G(\phi)} Z_n(\phi) = \infty$
and {\em transient} otherwise. If recurrent, $\phi$ is called {\em positive recurrent}
if $\sum_n n e^{-nP_G(\phi)} Z_n^*(\phi) < \infty$ and {\em null recurrent} otherwise.

A measure $m$ on $Y$ is called \emph{$\phi$-conformal} if $m(F(A))=\int_A e^{-\phi}dm$ 
whenever $A$ is measurable and $F$ is injective on $A$.  
Moreover, a measure $\nu$ on $Y$ is called \emph{conservative} if any measurable set $W\subset Y$ such that 
the sets $\{F^{-n}(W)\}_{n=0}^{\infty}$ are disjoint has $\nu(W)=0$.

Sarig \cite[Theorem 2]{Sarig01} generalises the Ruelle-Perron-Frobenius (RPF) Theorem to countable Markov shifts,
assuming that $\phi \in SV$ and $P(\phi)=\log\lambda<\infty$. Then, in the BIP setting, $\phi$ is necessarily  positive recurrent and for the Perron-Frobenius operator $(L_\phi v)(x) = \sum_{Fy=x} e^{\phi(y)} v(y)$, 
 there exists a conservative $\phi$-conformal measure $m$ and a continuous function $h$ such 
that $L_{\phi}^{*} m= \lambda m$, $L_{ \phi} h= \lambda h$ and  $ h~{ d} m < \infty$.  
We will call the measure $h\,dm$ coming from the RPF theorem, the \emph{RPF measure} $\mu_\phi$.  Here 
$\mu_\phi$ also has the {\em Gibbs property}, i.e., 
$\mu_\phi(C_n) \asymp e^{S_n\phi(x) - nP(\phi)}$
 for all $n \in \N$ and all non-empty $n$-cylinders $C_n$ and any $x\in C_n$.
Moreover,  \cite[Theorem 2]{Sarig01} implies that whenever $\mu_\phi$  has finite entropy, then it is an 
equilibrium measure for $\phi$.

If we start from a general dynamical system $f:X\to X$ and potential $\phi: X\to [-\infty, \infty)$, then this 
may not be Markov.
We extend the notion of positive/null recurrent to this case, assuming that our system `induces' to a 
Markov system (see below). Then, as in \cite{IomTod10}, $(X,f, \phi)$ 
is called \emph{recurrent} if there is a conservative $(\phi-P(\phi))$-conformal measure $m$ which comes 
from the induced system. Moreover, if  there exists a finite $f$-invariant measure $\mu\ll m$, then we say 
that $\phi$ is \emph{positive recurrent}; otherwise we say that $\phi$ is \emph{null recurrent}.

\section{Abstract set-up}\label{sec:abstsu}

The following subsection contains assumptions on our dynamics and our `base potential' $\phi$ 
which we will assume throughout the theoretical sections, i.e., Sections~\ref{sec:abstsu}--\ref{sec-correl}.

\subsection{Basic assumptions on $f$ and $\phi$}
\label{ssec:basic}

Let $f:X\to X$ and assume that $\phi:X\to [-\infty, \infty)$ is a null recurrent potential for $f$.  
We assume that $P(\phi)=0$ (otherwise replace $\phi$ with $\phi-P(\phi)$). 
We assume that $\mu_\phi$ is an infinite equilibrium measure for $(f,\phi)$ in the sense of~\cite{Sarig01}, 
so $\mu_\phi(X)=\infty$ .  As recalled below,  `infinite equilibrium measures' are known to 
be meaningful when $f$ induces, with some general (i.e., not necessarily first) return time, to a BIP Markov map. 

Fix $Y \subset X$ such that $\mu_\phi(Y)\in (0,\infty)$. 
Let $\tau:Y\to\N$ 
be a general return time and define the return map $F=f^\tau:Y\to Y$.
From here on we will assume that $F$ satisfies the BIP property and let $\mathcal{A}$ 
denote the corresponding Markov partition.
Let $\bar\phi=\sum_{j=0}^{\tau-1}\phi\circ f^j$ be the induced version of $\phi$.  We will assume that $\bar\phi\in SV$.
Since $\bar\phi$ is positive recurrent, $\mu_{\bar\phi}$ is a finite equilibrium measure for $(F,\bar\phi)$ 
(see, for instance,~\cite{Sarig01}).
It is common knowledge that if $\tau$ is a \emph{first} return time,  a sigma-finite invariant measure $\mu_\phi$ 
(finite or infinite) for the original system 
can be obtained by pulling back $\mu_{\bar\phi}$ (as in \eqref{eq:pullback}).
In order to ensure that the same holds when $\tau$ is a general return with 
$\int_Y\tau\, d\mu_{\bar\phi}=\infty$, we further assume that there exists a 1-cylinder $Y_0$ and a
reinduced time $\rho:Y_0\to\N$,
such that if ${\r}:Y_0\to\N$ is a first return to $Y_0$, then we have
\begin{equation}\label{eq-reind}
f^\tau=(f^{\r})^\rho\mbox{ with } \int_{Y_0}\rho\, d\mu_{\phi}<\infty.
\end{equation}
When $\int_Y\tau\, d\mu_{\bar\phi}=\infty$, assumption~\eqref{eq-reind} ensures that an infinite, 
sigma-finite invariant measure $\mu_\phi$ for the original system can be obtained 
by pulling back $\mu_{\bar\phi}$ (see~\cite[Theorem 2.1]{BNT} and \cite[Theorem 1.1]{Zwe05}): 
for any measurable set $A$, 
\begin{equation}\label{eq:pullback}
\mu_\phi(A) = \sum_{k \geq 0} \mu_{\bar\phi}(f^{-k}(A) \cap \{ \tau > k\}).
\end{equation}
When $\tau$ is the first return time to $Y$, \eqref{eq:pullback} gives that
$\mu_\phi(Y) = \mu_{\bar\phi}(Y)$.  As natural in an infinite measure setting, 
when $\int \tau \ d\mu_{\bar\phi} = \infty$,  we will not normalise $\mu_\phi$.

\subsection{Liftability}

Since we will be taking information from inducing schemes to learn about our original system, we need to ensure that the given inducing scheme is compatible with the relevant potentials and measures.  This is `liftability':

\begin{defn}
We call a measure for the original system \emph{liftable} if it can be obtained from an 
induced measure via \eqref{eq:pullback}. 

For a potential $\chi: X\to [-\infty, \infty)$, we say that  $(Y,F)$ is   $\chi$-\emph{liftable} if 
$P(\overline{\chi-P(\chi)})=0$.
\end{defn}

Note that if $\chi$ has an equilibrium measure $\mu_\chi$ which lifts to $(Y, F)$ as in  \eqref{eq:pullback}, 
then  $(Y,F)$ is  $\chi$-liftable.  This follows since $h(\mu_{\chi})+\int(\chi-P(\chi))~d\mu_{\bar\chi}=0$ and $ \int\tau~d\mu_{\bar\chi}<\infty$, 
so Abramov's formula, see \cite[Section 5]{Zwe05}, gives that the induced measure $\mu_{\bar\chi}$ has 
$$h(\mu_{\bar\chi})+\int\overline{\chi-P(\chi)}~d\mu_{\bar\chi}= \left(\int\tau~d\mu_{\bar\chi}\right)\left(h(\mu_{\chi})+\int\left(\chi-P(\chi)\right)~d\mu_{\bar\chi}\right)=0.$$
Since we always have $P(\overline{\chi-P(\chi)})\le 0$, see e.g.\ Lemma~\ref{lem:induced p nonpos} below, this means that
$P(\overline{\chi-P(\chi)})=0$, so $(Y,F)$ is  $\chi$-liftable.
 
\subsection{Assumptions on tails of $\mu_{\bar\phi}$ and on $\psi$: (H1) and (H2)}\label{sec:psi}

Estimates on the tails $\mu_{\bar\phi}(\tau>n)$ are essential for our results.
For several arguments in this work, we require:
\begin{itemize}
\item[\bf (H1)(a)] $C_2n^{-\beta}\le \mu_{\bar\phi}(\tau>n)\le C_1n^{-\beta}$ with $\beta\in (0,1)$ 
and $C_1\geq C_2 > 0$. 
\end{itemize}
In particular, (H1)(a) implies that $\int_Y \tau\, d\mu_{\bar\phi}=\infty$.

We let $L_{\mu_\phi}$ and $R_{\mu_ {\bar\phi}}$ the normalised transfer operators defined 
w.r.t.\ $\mu_\phi$ and $\mu_ {\bar\phi}$, respectively.
Recall that the transfer operator $R_{\mu_{\bar\phi}}:L^1(\mu_{\bar\phi})\to L^1(\mu_{\bar\phi})$ is given
by 
$$
\int_Y R_{\mu_{\bar\phi}}v\,w\,d\mu = \int_Y v\,w\circ F\,d\mu \ \text{ for } w\in L^\infty(\mu_{\bar\phi}).
$$ 
A similar definition holds for $L_{\mu_\phi}$.

We are interested in the asymptotic behaviour of the pressure function $P(\phi+s\psi)$
as $s \to 0$ and a potential $\psi:X\to[-\infty, \infty)$ 
with induced version $\bar\psi=\sum_{j=0}^{\tau-1}\psi\circ f^j$
 satisfying certain assumptions (see (H2) below).  
 As we will see in Section~\ref{sec:ev}, such assumptions together with 
 good functional analytic properties of the 
induced system $(Y, F, \mu_{\bar\phi})$ (see (P1) and (P2) below), will allow us to speak 
of  the family of leading eigenvalues $\lambda(u,s)$ associated with the perturbed family of operators
\begin{equation}\label{eq-RUS}
R(u,s)v:=R_{\mu_{\bar\phi}}(e^{-u\tau}e^{s\bar\psi}v), \qquad  u \ge 0, s\in (0,\delta_0),
\end{equation}
for some $\delta_0 > 0$ 
and identify $P(\overline{\phi+s\psi-u})$ with $\log\lambda(u,s)$ for 
$u \in [0,\delta_0)$ and $s \in (0,\delta_0)$. 

The following stronger version of (H1)(a) allows for a good understanding  
of the asymptotics of $P(\overline{\phi+s\psi-u})$ as $u, s\to 0$ (see Corollary~\ref{cor-deriveign}).
\begin{itemize}
\item[\bf (H1)(b)] 
$\mu_{\bar\phi}(\tau(y)>n)=c n^{-\beta}+ b(n)+H(n)$, 
for some $\beta\in(0,1)$, $c > 0$ and some function $b$ such that $nb(n)$ has bounded variation 
and $b(n)=O(n^{-2\beta})$, and $H(n)=O(n^{-\xi})$ with $\xi>2$.
\end{itemize}

Throughout, we let $R(u)v:= R_{\mu_{\bar\phi}}(e^{-u\tau}v)$
and for each $n\ge1$, we define
$R_n:L^1(\mu_{\bar\phi})\to L^1(\mu_{\bar\phi})$ by $R_nv=R_{\mu_{\bar\phi}}(1_{\{\tau=n\}}v)$.
It is easily verified that $R(u)=\sum_{n=1}^\infty R_n e^{-un}$.

With the above quantities defined we can recall some functional analytic properties of the 
induced Markov  BIP map $(Y, F, \mathcal{A}, \mu_{\bar\phi})$.
Under the assumptions of Section~\ref{ssec:basic} and (H1)(a), there is a Banach space  $\B$ of bounded piecewise H{\"o}lder functions compactly 
embedded in $L^\infty(\mu_{\bar\phi})$ 
under which the properties (P1) and (P2) below hold; see~\cite{AaronsonDenker01} for (P2) and~\cite{Sarig02} for (P1). 
The norm on $\B$ is defined by $\| v \|_\B = |v|_\theta + |v|_\infty$,
where $|v|_\theta = \sup_{a \in \mathcal{A}} \sup_{x \neq y \in a} |v(x)-v(y)| / d_\theta(x,y)$,
where $d_\theta(x,y) = \theta^{s(x,y)}$ for some $\theta\in (0,1)$,
and $s(x,y) = \min\{ n : F^n(x) \text{ and } F^n(y) \text{ are in different elements of } \mathcal{A}\}$ 
is the separation time.

\begin{itemize}
\item[\bf (P1)] For all $n\ge1$,
$R_n:\mathcal{B}\to\mathcal{B}$ is a bounded linear 
operator with $\sum_{j>n}\|R_j \|=O(n^{-\beta})$ with $\beta$ as in (H1)(a).
\end{itemize}

We notice that $u\mapsto R(u)$ is an analytic family of bounded linear 
operators on $\mathcal{B}$ for $u> 0$.  This implies that there exists
$\delta_0>0$  such that the  family of associated eigenvalues $\lambda(u)$, 
$u\in B_{\delta_0}(0)$, is well defined and analytic.

Since $R(0)=R_{\mu_{\bar\phi}}$ and $\mathcal{B}$ contains constant functions, 
$1$ is an eigenvalue of $R(0)$. 
\begin{itemize}
\item[\bf (P2)] The eigenvalue $1$ is simple and isolated
in the spectrum of $R(0)$.
\end{itemize}

Finally, we formulate our assumptions on the induced version $\bar\psi:Y\to\R$ of 
the potential $\psi:X\to[-\infty, \infty)$.

\begin{itemize}
\item[\bf (H2)] The induced potential $\bar\psi\in SV$ and $(Y, F)$ is 
$(\phi+s\psi)$-liftable for $s\in [0, \delta)$ for some $\delta \in (0,\delta_0)$
with $\delta_0$ from \eqref{eq-RUS}.
Moreover, one of the following holds:
\begin{itemize}
\item[\bf (i)] $|\bar\psi|\in L^1(\mu_{\bar\phi})$ and there exist $\eps>0$
such that for all $s\in (0,\delta)$,
\[
\|R_n (e^{s\bar\psi}-1)\|\ll s\,\|R_n\|\, n^{\eps};
\]
\item[\bf (ii)] $\bar\psi(x)=C'-\psi_0$, where $0 \leq \psi_0(x) \le C \tau^\gamma(x)$, 
for $C', C > 0$ and $\gamma\in [\beta,1]$. Also $\psi_0$ is piecewise  H{\"o}lder 
(with exponent $\theta$ as in the definition of the Banach space $\B$).
\end{itemize}
\end{itemize}

\begin{rmk}
\label{rmk-log} In our examples, we will make certain assumptions on $\psi$ and show that they imply (H2). 
In particular when $\psi$ is the natural potential $\log|f'|$, then this will give 
$|\bar\psi|=|\log|F'|| \leq \xi\log\tau$ for some $\xi > 0$, and here (H2)(i) will apply.
Indeed, for every $s \in (0, \eps/(2\xi)]$ we have $\xi n^{u\xi} \log n < n^\eps$ for 
$n$ sufficiently large and all $0 < u \leq s$.
It follows that
$$
e^{s \xi \log n}-1 = n^{s \xi} - 1 = \int_0^s \frac{d}{du} n^{u\xi} \, du
= \int_0^s \xi \ n^{u\xi}\ \log n  \, du < \int_0^s n^{\eps} \, du = s n^\eps,
$$
as required.
\end{rmk}

It may be useful at this stage to note that conditions (H1)(a) (and thus (P1)) and (H2) imply that $\phi+s\psi$ is recurrent, as shown below in Lemma~\ref{lem:rec}.

\section{Results on the asymptotics of the pressure function in the abstract set-up}
\label{sec:presfnc}

In this section we obtain the asymptotics of the pressure function 
$P(\phi+s\psi)$ under (H1)(a) or (H1)(b), in the setting of Section~\ref{ssec:basic}.
Note that if $\sup\psi<\infty$ then $P(\phi+s\psi) <\infty$. 

The first result below gives the higher order asymptotics under the strong assumption (H1)(b).  

\begin{thm}\label{thm:PbarP} 
Assume $P(\phi)= 0$, $\bar\phi \in SV$ and \eqref{eq-reind}.
Let $\psi:X\to[-\infty, \infty)$ be bounded from above (we allow $\psi$ to be unbounded from below)
such that $P(\phi+s\psi) > 0$ for $s > 0$
and assume that the induced version $\bar\psi$ (of $\psi$) satisfies (H2). 
Assume that\footnote{As in  Lemma~\ref{lem:flat_pres}, in many natural settings it can be shown 
that $P(\phi+s\psi)=o(s)$, and thus this is a very mild assumption.}  
there exists $a>0$ such that $s\ll (P(\phi+s\psi))^a$, as $s\to 0$.

Assume  (H1)(b) and recall that (P1) and (P2) hold. Set $C=(c \beta\Gamma(1-\beta))^{-1}$ with $c>0$ as in (H1)(b).
Then, there exists a constant $C_0>0$ such that\footnote{The precise form of this constant  is given inside the proof.} 
 $$
P(\phi+s\psi) = \left(C P(\overline{\phi+s\psi})\right)^{1/\beta}(1+Q(s)) \text{ as } s \to 0,
$$
where $Q(s) = - C_0 C^{1/\beta} P(\overline{\phi+s\psi})^{(1-\beta)/\beta} + 
O(P(\overline{\phi+s\psi})^{\beta a-\eps})$, for arbitrarily small $\eps>0$.
\end{thm}

The next result gives the higher order asymptotics under the mild assumption (H1)(a).
\begin{thm}\label{thm:PbarP2} 
Assume the setting of Theorem~\ref{thm:PbarP} with (H1)(a) instead of (H1)(b). Let $C_1, C_2>0$ as 
in (H1)(a) and set $C_1^*=(C_1 \beta\Gamma(1-\beta))^{-1}$, $C_2^*=(C_2 \beta\Gamma(1-\beta))^{-1}$
Then
 $$
\left(C_2^* P(\overline{\phi+s\psi})\right)^{1/\beta}(1+Q(s)) \le P(\phi+s\psi) 
\le \left(C_1^* P(\overline{\phi+s\psi})\right)^{1/\beta}(1+Q(s)) \text{ as } s \to 0,
$$
where $Q(s)\to 0$, as $s\to 0$.
\end{thm}

Note that under the two different forms of behaviour given in (H2), one can derive further information on the asymptotics of $s\mapsto P(\overline{\phi+s\psi})$, see Remark~\ref{rmk-H2} below.

The proofs of Theorems~\ref{thm:PbarP} and \ref{thm:PbarP2} are given in 
Section~\ref{ssec:main_thm_pfs}, following the proofs of some more technical results in 
Section~\ref{ssec:tech props proofs}.

\section{Family of eigenvalues associated with $R(u,s)$ and asymptotics of the induced pressure}
\label{sec:ev}

In this section we first give results on the asymptotics of $P(\overline{\phi+s\psi-u})$, 
as $u, s \to 0$, via the asymptotic behaviour of family of eigenvalues associated with 
$R(u,s)$ (defined in~\eqref{eq-RUS}): see Corollary~\ref{cor-deriveign}.
We first justify that this family of eigenvalues associated with $R(u,s)$ is well defined (and continuous) 
in a neighbourhood of $(0,0)$.
We start with the following continuity properties of $R(u)$ and $R(u,s)$.

The first result below for $R(u)$  is standard (see, for instance,~\cite[Lemma 3.1]{Gouezel04}). 
The second is an analogous version for $R(u,s)$ obtained under (H2).

\begin{lemma}\label{lemma-P1}
Assume (H1)(a) and recall that (P1) holds, and take $\delta_0$ as in \eqref{eq-RUS}.
Then there is a constant $C>0$ such that for all $u\in (0, \delta_0)$,  
$\|R(u)-R(0)\|\le C  u^\beta$.
Moreover, the same estimates  are inherited by the
families $\lambda(u)$ and $v(u)$, where defined.
\end{lemma}

\begin{lemma}\label{lemma-rus} 
Assume (H1)(a) and recall that (P1) holds. Suppose that (H2) holds for some $\delta\le \delta_0$. 
Then there exists $C>0$ such that
for all $u>0$, for all $s\in (0,\delta)$ and for any $\eps<\beta$,
\[
\|R(u,s)-R(u,0)\|\le C\,s^{\beta-\eps}.
\]
\end{lemma}

\begin{proof} Recall that the norm on $\B$ is $\| v \|_\B = |v|_\theta + |v|_\infty$,
where $|v|_\theta = \sup_{a \in \mathcal{A}} \sup_{x \neq y \in a} |v(x)-v(y)| / d_\theta(x,y)$.
If (H2)(i) holds,
\begin{align*}\label{eq-Rus0}
\|R(u,s)-R(u,0)\|=\left\|\sum_{n\geq 1} R_n(e^{s\bar\psi}-1)e^{-un}\right\|
\le sC\sum_{n\geq 1} \|R_n\|n^{\beta'}
\end{align*} 
for some $C>0$ and any $\beta'<\beta$. Next, let $S_n=\sum_{j\ge n}\|R_j\|$ and note that by (P1), 
$S_n\le C n^{-\beta}$ for some $C>0$.
Since $\sum_{j\ge n}\|R_j\|=C'<\infty$, we compute that 
\begin{align*}
\sum_{n\geq 1} \|R_n\|n^{\beta'}=\sum_{n\geq 1} (S_n-S_{n-1})n^{\beta'}=&\sum_{n\geq 1}S_n(n^{\beta'}-(n-1)^{\beta'})+\sum_{j\ge n}\|R_j\|\\
&\le C\sum_{n\geq 1}n^{-\beta}n^{\beta'-1}+C'<\infty.
\end{align*}
Hence in this case the conclusion follows with a better than stated estimate. 

If (H2)(ii) holds, we recall that $\psi_0$ is a non-negative piecewise H{\"o}lder function with 
$\psi_0 \ll \tau^\gamma$, $\gamma\in [\beta,1]$.
Note that $|s\bar\psi|$ is bounded on $\{\tau=n\}$ for $n$ small
and $|e^{s\bar\psi}-1|\ll s^{\beta-\eps} n^{\gamma(\beta-\eps)}$ for $n$ large.
Using that $\psi_0$ is  piecewise H{\"o}lder, for all $x,y\in \{\tau=n\}$ with $n$ large, 
$|e^{s\bar\psi(x)}-e^{s\bar\psi(y)}|\ll s^{\beta-\eps} n^{\gamma(\beta-\eps)} d_\theta(x,y)$.
Recall $s\in (0,\delta)$ for small enough $\delta>0$. Putting the above together,  for any $\eps<\beta$ and some $C>0$,
\[
 \|R_n(e^{s\bar\psi}-1)\|\le C (s+s^{\beta-\eps} n^{\gamma(\beta-\eps)})\|R_n\|.
\]
The above inequality and a repeat  of the argument used in the case when (H2)(i) holds leads to
\[
\|R(u,s)-R(u,0)\|\le C s^{\beta-\eps}  \sum_{n\geq 1}n^{-\beta}n^{\gamma(\beta-\eps)-1},
\] 
and the conclusion follows since $\gamma(\beta-\eps)<\beta$.
\end{proof}

Recall that under (H2)(a), (P2) holds and so $\lambda(0)=1$ is a simple isolated eigenvalue in the spectrum of $R(0)$, and that 
$u\mapsto R(u)$ is analytic in $u$ for $u\in (0, \delta_0)$. 
This together with Lemma~\ref{lemma-rus} implies that $(u,s)\mapsto R(u,s)$ is analytic 
in $u$, $u>0$, and $C^{\beta-\eps}$ in $s$, $s\in(0, \delta_0)$.
Thus, there exists a family of simple eigenvalues $\lambda(u,s)$,  $s\in [0,\delta_0)$, 
analytic in $u\in (0, \delta_0)$ and  $C^{\beta-\eps}$ in $s$ with $\lambda(0,0)=\lambda(0)=1$.  
By the RPF Theorem, $\log \lambda(u, s) =P(\overline{\phi+s\psi-u}).$ In the setting of the following 
lemma we obtain $\log\lambda(P(\phi+s\psi), s)= P(\overline{\phi+s\psi-P(\phi+s\psi)})=0$ for all small $s$. 

\begin{lemma}
Assume (H1)(a) and recall that (P1) holds.  Assume (H2) and suppose that $P(\phi+s\psi)>0$ for $s > 0$. Then
 $\phi+s\psi$ is recurrent for all $s \in [0,\delta_0)$ and positive recurrent for all $s \in (0, \delta_0)$. 
\label{lem:rec}
\end{lemma}

\begin{proof}
By (H2), $P(\overline{\phi+s\psi-P(\phi+s\psi)})=0$, so the recurrence of $\phi+s\psi$ follows if 
we can project the (conservative) $\overline{\phi+s\psi-P(\phi+s\psi)}$-conformal measure.  
If $F$ is a first return map then one can project this measure in the natural way (for work on 
projecting conformal measures see for example the appendix of \cite{IT13}). 

If $F$ is not a first return map, then it suffices to show that \eqref{eq-reind} holds 
for this conformal measure, i.e., the reinducing time $\rho$ is integrable.  By the RPF Theorem, we have 
an RPF measure for $\overline{\phi+s\psi-u}$, which we will denote by $\mu_{u, s}$.
Note that the density of $\mu_{u, s}$ is uniformly bounded on $Y_0$ (from the BIP property), so it suffices to show 
$\int_{Y_0}\rho~d\mu_{u, s}<\infty$.  

Let $\{ Y_i \}_{i \geq 1}$ be the Markov partition of the induced map.
The Gibbs property, used for $\mu_{u,s}$ and $\mu_{0,s}$ respectively, gives
\begin{eqnarray*}
\int_{Y_0}\rho~d\mu_{u, s} &\leq & \int_{Y_0}\tau~d\mu_{u, s}=\sum_n n \mu_{u,s}(\tau=n) 
\asymp \sum_n n \sum_{\tau(Y_i)=n} e^{S_n(\phi+s\psi)-nu} \\
&=& \sum_n ne^{-nu} \sum_{\tau(Y_i)=n} e^{S_n(\phi+s\psi)} \asymp \sum_nne^{-nu}\mu_{0, s}(\tau=n) < \infty.
\end{eqnarray*}
So setting $u=P(\phi+s\psi)$, which is in $(0, \delta_0)$ for $s$ sufficiently small, we can indeed project 
our $\overline{\phi+s\psi-P(\phi+s\psi)}$-conformal measure to a 
conservative $(\phi+s\psi-P(\phi+s\psi))$-conformal measure 
\end{proof}

The first result below gives the asymptotic behaviour of $\lambda(u,s)$.
To state it we need the following

\textbf{Notation:}  Let $c$, $b(n)$ and  $H(n)$ be as given in (H1)(a). 
Let $H_1(x)=c([x]^{-\beta}-x^{-\beta})+b([x])+ H([x])$, where by $[.]$
stands for the ceiling function.
With the convention $0^{-\beta}=0$, the function $H_1(x)$ is well defined in $[0,1)$ and we set 
$c_{H}=\int_0^\infty H_1(x)\,dx$.  

\begin{prop}\label{prop-deriveign}
Assume (H1)(b) and recall that (P1)  and (P2) hold. Suppose that  (H2) holds.   
Set $\Pi(s)=\int_Y(e^{s\bar\psi}-1)\, d\mu_{\bar\phi}$. 
Then the following holds as $u,s\to 0$.
\begin{equation*}\label{eq:lambda}
1-\lambda(u,s)=c \Gamma(1-\beta) u^{\beta} + c_H u+E(u)+\Pi(s)+ D(u,s),
\end{equation*}
where  $E(u)=O(u^{2\beta})$, $D(u,s)=o(u^\beta+\int_Y(e^{s\bar\psi}-1)\, d\mu_{\bar\phi})$ and for 
arbitrarily small $\eps>0$,
\[
\frac{d}{du} E(u)=O(u^{2\beta-1}), \quad \frac{d}{du}D(u,s)\ll s^\eps\, u^{(\beta-\eps-1)}( u^\beta+s^{\beta-\eps}).
\]
\end{prop}

An immediate consequence of the above result is:

\begin{cor}\label{cor-deriveign}
Assume the setting of Proposition~\ref{prop-deriveign}. Moreover, assume that there exists  $a>0$ such that $s=O(u^a)$ as $u\to 0$.
Then for arbitrarily small $\eps>0$, 
\begin{equation*}
\frac{d}{d u}P(\overline{\phi+s\psi-u})=-c \beta\Gamma(1-\beta) u^{\beta-1} -c_H+O(u^{2\beta-1})+O(s^\eps\, u^{2\beta-\eps-1})+O(s^{\beta-\eps/a}u^{\beta-1}).
\end{equation*}
\end{cor}
\begin{proof}By Proposition~\ref{prop-deriveign}, $\lambda(u,s)=1+g(u,s)$, where $g(u,s)\to 0$ as $u,s\to 0$ and for arbitrarily small $\eps>0$,
\[
\frac{d}{d u}\lambda(u,s)=-c \beta\Gamma(1-\beta) u^{\beta-1} + c_H+O(u^{2\beta-1})+ O(s^\eps\, u^{2\beta-\eps-1})+O( s^{\beta}u^{\beta-\eps-1}).
\]
Clearly, $O(s^\eps\, u^{2\beta-\eps-1})=o(u^{\beta-1})$, so this gives an error term. For the term $O( s^{\beta}u^{\beta-\eps-1})$, using that
$s=O(u^a)$ for some $a>0$ and choosing $\eps<a\beta$, we have $s^{\beta}u^{\beta-\eps-1}\ll s^{\beta-\eps/ a}u^{\beta-1}$. 
The conclusion follows since $P({\overline{\phi+s\psi-u}}) = \log \lambda(u,s)$
by Lemma~\ref{lem:rec}.
\end{proof}

The next two results give the expansion and derivative in $u$ of the eigenvalue $\lambda(u,s)$
and pressure $P(\overline{\phi+s\psi-u})$  under (H1)(a) (a much weaker assumption that (H1)(b)).

\begin{prop}\label{prop-deriveign2}
 Assume (H1)(a) and recall that (P1)and (P2) hold. Suppose that (H2) holds.  
 Set $\Pi(s)=\int_Y(e^{s\bar\psi}-1)\, d\mu_{\bar\phi}$. Then as $u,s\to 0$.
\begin{equation*}
C_2 \Gamma(1-\beta) u^{\beta}(1+B(u))\le 1-\lambda(u,s)- \Pi(s)+ D(u,s)\le C_1 \Gamma(1-\beta) u^{\beta}(1+B(u)),
\end{equation*}
where  $C_1, C_2$ are as in (H1)(a), $B(u)\to 0$ as $u\to 0$ and $D(u,s)=o(u^\beta+\int_Y(e^{s\bar\psi}-1)\, d\mu_{\bar\phi})$. Moreover,
\begin{equation*}
-C_2\beta \Gamma(1-\beta) u^{\beta-1} (1+E(u))\le \frac{d}{du}\lambda(u,s)+\frac{d}{du}D(u,s)\le -C_1 \beta\Gamma(1-\beta) u^{\beta-1}(1+E(u)),
\end{equation*}
where $E(u)\to 0$ as $u\to 0$ and
$\frac{d}{du}D(u,s)\ll s^\eps\, u^{(\beta-\eps-1)}( u^\beta+s^{\beta-\eps})$, for arbitrarily small $\eps>0$.
\end{prop}

\begin{cor}\label{cor-deriveign2}
Assume the setting of Proposition~\ref{prop-deriveign2}. Moreover, assume that there exists $a>0$ 
such that $s=O(u^a)$ as $u\to 0$. 
Then for arbitrarily small $\eps>0$, as $u\to 0$,
\begin{align*}
 -C_2\beta \Gamma(1-\beta) u^{\beta-1} (1+E(u))&\le \frac{d}{du}P(\overline{\phi+s\psi-u})+\frac{d}{du}D(u,s)\\
&\le -C_2\beta \Gamma(1-\beta) u^{\beta-1} (1+E(u)),
\end{align*}
where the functions $E(u)$ and $\frac{d}{du}D(u,s)$ are as in the conclusion of  Proposition~\ref{prop-deriveign2}. 
\end{cor}

\begin{proof}
The conclusion follows by the argument used in Corollary~\ref{cor-deriveign} using 
Proposition~\ref{prop-deriveign2} instead of Proposition~\ref{prop-deriveign}.~\end{proof}

\subsection{Proof of Propositions~\ref{prop-deriveign} and~\ref{prop-deriveign2}}
\label{ssec:tech props proofs}

\begin{pfof}{Proposition~\ref{prop-deriveign}} The proof below is a version of the argument
used in~\cite[Proof of Proposition 2.6]{T15} simplified by the fact that we only need to deal with
real perturbations ($e^{-u\tau}$ as opposed to $e^{-(u+i\theta)\tau}$, $\theta\in [-\pi,\pi)$).
However, we need to spell out the argument due to the second perturbation $e^{s\bar\psi}$.

For $u \in (0,\delta)$ write $R(u,s)v(u,s)=\lambda(u,s)v(u,s)$. Normalise such that $\int_Y v(u,s)\, d\mu_{\bar\phi}=1$. 
Integrating both sides and using the formalism in~\cite{Gouezel10} 
(a simplification of~\cite{AaronsonDenker01}), write 
\begin{align}\label{eq-ev-Gou}
\nonumber 1-\lambda(u,s)& =1-\int_Y\lambda(u,s)v(u,s)d\mu_{\bar\phi}=1-\int_Y R(u,s)(v(u,s))d\mu_{\bar\phi}\\
&=\int_Y(1-e^{-u\tau}e^{s\bar\psi})\, d\mu_{\bar\phi}-V(u,s),
\end{align}
where $V(u,s)=\int_Y (R(u,s)-R(0))(v(u,s)-v(0))\,d\mu_{\bar\phi}$. 
By Lemmas~\ref{lemma-P1} and~\ref{lemma-rus}, $\|R(u,s)-R(0,0)\|\ll u^\beta+s^{\beta-\eps}$, for any $\eps<\beta$. 
The same holds for  $\|v(u,s)-v(0,0)\|$. Since $\B\subset L^\infty(\mu_{\bar\phi})$, 
\begin{equation}
\label{eq-vus}
|V(u,s)|\ll  (u^\beta+s^{\beta-\eps})^2.
\end{equation}

Next,  we estimate  $|\frac{d}{du} V(u,s)|$. First, note that $\|\frac{d}{du} R(u,s)\|=\| \sum_{n\ge 1}n R_n(e^{s\bar\psi}-1)e^{-un}\|$. If (H2(i) holds, then
$\|R_n(e^{s\bar\psi}-1)\| \ll s n^{\eps}\|R_n\|$ for any $\eps>0$ and thus,
\[
\left\|\frac{d}{du} R(u,s)\right\| \ll \sum_{n\ge 1}n^{1+\eps}\|R_n\|e^{-un}.
\]
Let $S_n=\sum_{j\ge n}\|R_j\|$ and note that by (P1), $S_n\ll  n^{-\beta}$.
Since $\sum_{j\ge n}\|R_j\|<\infty$,
\begin{align*}
\sum_{n\geq 1} &\|R_n\|n^{1+\eps}e^{-un}=\sum_{n\geq 1} (S_n-S_{n-1})n^{1+\eps}e^{-un}=\sum_{n\geq 1}S_n(n^{1+\eps}-(n-1)^{1+\eps})e^{-un}\\
&+\sum_{n\ge 1}(\sum_{j\ge n}\|R_j\|)e^{-un}
\ll \sum_{n\geq 1}n^{-(\beta-\eps)}e^{-un}\ll\int_0^\infty x^{-(\beta-\eps)}e^{-ux}\, dx\ll u^{(\beta-\eps-1)}.
\end{align*}
Altogether, if (H2)(i) holds then $\|\frac{d}{du} R(u,s)\| \ll s\, u^{(\beta-\eps-1)}$ and the same holds  for $\|\frac{d}{du} v(u,s)\|$.  We already know that $\|R(u,s)-R(0,0)\|\ll u^\beta+s^{\beta}$
and that the same holds for  $\|v(u,s)-v(0,0)\|$. Thus, for arbitrarily small $\eps>0$,
$|\frac{d}{du} V(u,s)|\ll  s\, u^{(\beta-\eps-1)}( u^\beta+s^{\beta-\eps})$.

If (H2)(ii) holds then for arbitrarily small $\eps>0$, $\|R_n(e^{s\bar\psi}-1)\|\ll  s\|R_n\|+ s^{\eps/\gamma} n^{\eps}\|R_n\|$, for $\gamma\in[\beta,1]$ and
proceeding in the case of (H2)(i), we obtain that   $\|\frac{d}{du} R(u,s)\| \ll s^\eps\, u^{(\beta-\eps-1)}$. 
Thus, taking the worst estimate in the (H2)(i) and (H2)(ii), for arbitrarily small $\eps>0$,
\begin{equation*}
\label{eq-der-op}
\left|\frac{d}{du} V(u,s)\right|\ll  s^\eps\, u^{(\beta-\eps-1)}( u^\beta+s^{\beta-\eps}).
\end{equation*}

In the rest of the proof, we deal with the pure scalar part and write
\begin{align*}
Q(u,s) &:=\int_Y(1-e^{-u\tau}e^{s\bar\psi})\, d\mu_{\bar\phi}\\
&=\int_Y(1-e^{-u\tau})\, d\mu_{\bar\phi}+\int_Y(1-e^{s\bar\psi})\, d\mu_{\bar\phi} 
- \int_Y(1-e^{-u\tau})(1-e^{s\bar\psi})\, d\mu_{\bar\phi}\\
&=\Psi(u)+\Pi(s)-W(u,s).
\end{align*}
It is easy to see that using H{\"o}lder inequality in either case of (H2), we 
have $|W(u,s)|\ll u^{\beta/2}s^{\beta/2}$. This estimate together with~\eqref{eq-ev-Gou} and~\eqref{eq-vus} 
gives the estimate for $D(u,s)$ in the statement of the proposition.
Next,
\begin{align*}
\frac{d}{du}Q(u,s)=\frac{d}{du}\Psi(u) - \frac{d}{du} W(u,s).
\end{align*}
By the argument used in~\cite[Proof of Proposition 2.6]{T15} (setting $\theta=0$ there and replacing 
$\frac{d}{d\theta}$ with $\frac{d}{du}$),
\[
\Psi(u)=c\Gamma(1-\beta) u^\beta + c_H u +E(u),
\]
where $|\frac{d}{du} E(u)|\ll u^{2\beta-1}$. To complete the proof we note that

\begin{align*}
\Big|\frac{d}{du} W(u,s)\Big|&\ll \Big|\int_Y\tau e^{-u\tau}(1-e^{s\bar\psi})\, d\mu_{\bar\phi}\Big|\ll s \int_Y R_{\mu_{\bar\phi}}(\tau e^{-u\tau}(1-e^{s\bar\psi}))\, d\mu_{\bar\phi}\\
&\ll \sum_n n \|R_n\|(1-e^{s\bar\psi}) e^{-un}.
\end{align*}

By the argument used in obtaining~\eqref{eq-vus}, we deduce that $|\frac{d}{du} W(u,s)|\ll s^\eps\, u^{(\beta-\eps-1)}( u^\beta+s^{\beta-\eps})$. 
The claimed estimate on $|\frac{d}{du} W(u,s)|$
follows from this together with~\eqref{eq-vus}, concluding the proof.~\end{pfof}

Using the same notation and some estimates obtained in the proof of Proposition~\ref{prop-deriveign}, we can complete the following.

\begin{pfof}{Proposition~\ref{prop-deriveign2}}
With the same notation used in the proof of Proposition~\ref{prop-deriveign}, write
\begin{align*}
\nonumber 1-\lambda(u,s)=\Psi(u)+\Pi(s)-W(u,s)-V(u,s),
\end{align*}
For $W(u,s)$ and $V(u,s)$ we have the estimates  from the proof of Proposition~\ref{prop-deriveign}, namely
\[
|W(u,s)|\ll u^{\beta/2}s^{\beta/2},\quad |V(u,s)|\ll  (u^\beta+s^{\beta-\eps})^2
\]
and
\[
 \left|\frac{d}{du} V(u,s)\right|,\, \left|\frac{d}{du} W(u,s)\right|
 \ll s^\eps\, u^{(\beta-\eps-1)}( u^\beta+s^{\beta-\eps}).
\]
To estimate  $\Psi(u), \frac{d}{du}\Psi(u)$ under (H1)(a), we need to write down the complete argument.
Define the distribution function $G(x)=\mu(\tau\le x)$ and note that under (H1)(a), $C_2 x^{-\beta} +C_2\Delta(x)\le 1-G(x)\le C_1 x^{-\beta} +C_1\Delta(x)$,
where $\Delta(x)=[x]^{-\beta}-x^{-\beta}=O(x^{-(\beta+1)})$.

Integration by parts
gives 
\begin{align*}
\Psi(u) &=\int_0^\infty (1-e^{-ux })\, dG(x) =-\int_0^\infty (1-e^{-ux })\, d(1-G(x))
=u\int_0^\infty e^{-ux }(1-G(x))\, dx.
\end{align*}
So,
\[
 C_2(K(u)+L(u))\le \Psi(u)\le C_1(K(u)+L(u)),
\]
where $K(u)=u\int_0^\infty e^{-ux }x^{-\beta}\, dx$ and $L(u)=u\int_0^\infty e^{-ux }\Delta(x)\, dx$. 
Here, we recall that $K(u)=\Gamma(1-\beta) u^\beta$.

First, let $C_\Delta=\int_0^\infty \Delta(x)\, dx$ and note that 
\begin{align*}
L(u)=uC_\Delta +u\int_0^\infty (e^{-ux }-1)\Delta(x)\, dx.
\end{align*}
But, $|\int_0^\infty (e^{-ux }-1)\Delta(x)\, dx|\ll u\int_0^{1/u} x\Delta(x)\, dx+\int_{1/u}^\infty \Delta(x)\, dx\ll u^\beta$. Hence,
$\int_0^\infty e^{-ux }\Delta(x)\, dx=A(u)$, where $A(u)=C_\Delta(1+o(1))$. Hence, $L(u)=uC_\Delta(1+o(1))$.
Altogether,
\[
 C_2\Gamma(1-\beta) u^\beta(1+B(u))\le \Psi(u)\le C_1\Gamma(1-\beta) u^\beta(1+B(u)),
\]
where $B(u)\to 0$ as $u\to 0$. Putting the above together, we obtain the asymptotics (within bounds) 
of $\lambda(u,s)$ as $u,s\to 0$.

Next, recall that  $A(u)=\int_0^\infty e^{-ux }\Delta(x)\, dx=C_\Delta(1+o(1))$ and thus
\begin{align*}
\frac{d}{du} L(u)=A(u)+u\int_0^\infty e^{-ux }x\Delta(x)\, dx=A(u)+M(u),
\end{align*}
where $|M(u)|\ll u\int_0^\infty e^{-ux } x^{-\beta}dx\ll u^\beta$. Moreover,
\begin{align*}
\frac{d}{du}K(u)&=\int_0^\infty e^{-ux }x^{-\beta}\, dx-u\int_0^\infty e^{-ux }x^{1-\beta}\, dx=\beta\int_0^\infty e^{-ux }x^{-\beta}\, dx\\
&=\beta\Gamma(1-\beta) u^{\beta-1}.
\end{align*}
Thus,
\[
 C_2\beta\Gamma(1-\beta) u^{\beta-1}(1+E(u))\le \frac{d}{du}\Psi(u)\le C_1\beta\Gamma(1-\beta) u^{\beta-1}(1+E(u)),
\] 
where $E(u)\to 0$ as $u\to 0$. The conclusion follows by putting the above together and using that
$-\frac{d}{du}\lambda(u,s)=\frac{d}{du}\Big(\Psi(u)-W(u,s)-V(u,s)\Big)$.
\end{pfof}

\subsection{Proof of Theorems~\ref{thm:PbarP} and \ref{thm:PbarP2}}
\label{ssec:main_thm_pfs}

\begin{pfof}{Theorem~\ref{thm:PbarP}}
Set $r(u,s) = \frac{d}{du} P(\overline{\phi + s \psi - u})$.
By Corollary~\ref{cor-deriveign} (with $c_H$ as defined there), the following holds as $u\to 0$ and $s=O(u^a)$ 
for some $a>0$ and for arbitrarily small $\eps>0$:
\begin{equation*}
r(u,s)=-c \beta\Gamma(1-\beta) u^{\beta-1} - c_H+O(u^{2\beta-1})+O(s^\eps\, u^{2\beta-\eps-1})
+O(s^{\beta-\eps/a}u^{\beta-1}).
\end{equation*}
For any small $u_0 > 0$, integration gives
\begin{align*}
P(\overline{\phi + s\psi-u_0}) -  P(\overline{\phi+s\psi})  
 =& \int_0^{u_0} r(u,s) \, du \\[1mm]
 =  -\ c \Gamma(1-\beta) u_0^{\beta} & - c_H u_0 +  O(u_0^{2\beta})+
 O(s^\eps\, u_0^{2\beta-\eps})+O(s^{\beta-\eps/a}u_0^{\beta}).
\end{align*}
By liftability, for $u_0 = u_0(s) = P(\phi+s\psi)$, 
we obtain $P(\overline{\phi+s\psi-u_0}) = 0$, so the left hand side  of the above expression becomes
$-P(\overline{\phi + s\psi)}$. By assumption, $u_0(s)>0$, for $s>0$. 
The continuity property of the pressure function gives $u_0(s)\to 0$ as $s\to 0$.
Thus,
\begin{equation}\label{eq:PbarP}
P(\overline{\phi+s\psi})=c \beta\Gamma(1-\beta) u_0^{\beta} + 
c_H u_0 +O(u_0^{2\beta})+O(s^\eps\, u_0^{2\beta-\eps})+O(s^{\beta-\eps/a}u_0^{\beta}),
\end{equation}
as $s\to 0$. Here $c>0$ comes from (H1)(b) and $c_H$ is real non-zero constant (see Proposition~\ref{prop-deriveign} 
for the exact form).
By assumption, there exists $a>0$ such that $s=O((u_0(s))^a)$.  
Hence, the above equation applies to $r(u_0,s)$. 
Recall that $C=(c \beta\Gamma(1-\beta))^{-1}$ and that $u_0(s) = P(\phi+s\psi)$. Hence,
\begin{eqnarray*}
P(\phi+s\psi) &=&
\left( CP(\overline{\phi+s\psi}) \right)^{1/\beta}\left(
1 - \frac{c_H}{\beta} C^{1/\beta} P(\overline{\phi+s\psi})^{(1-\beta)/\beta} \right.\\
&& \left. \qquad  
+\ O( P(\overline{\phi+s\psi})+ O(s^\eps\, P(\overline{\phi+s\psi})^{(\beta-\eps)/\beta})+O(s^{\beta-\eps/a}) \right),
\end{eqnarray*}
as $s \to 0$.  The conclusion follows since $s=O((u_0(s))^a)$.
\end{pfof}

\begin{pfof}{Theorem~\ref{thm:PbarP2}}
With the same notation as in the proof of Theorem~\ref{thm:PbarP}, under the assumption $s=O(u^a)$ 
for some $a>0$, Corollary~\ref{cor-deriveign2} gives that
\begin{align*}
 -C_2\beta \Gamma(1-\beta) u^{\beta-1} (1+E(u))&\le r(u,s)+\frac{d}{du}D(u,s)\\
&\le -C_2\beta \Gamma(1-\beta) u^{\beta-1} (1+E(u)).
\end{align*}
where $E(u)\to 0$ as $u\to 0$ and $\frac{d}{du}D(u,s)=O(s^\eps\, u^{2\beta-\eps-1})+O(s^{\beta-\eps/a}u^{\beta-1})$ 
for arbitrarily small $\eps>0$. The conclusion follows from this together with the argument in 
the proof of Theorem~\ref{thm:PbarP}.
\end{pfof}

\begin{rmk}
\label{rmk-H2}
From the statement of Proposition~\ref{prop-deriveign} we can see that the asymptotics of $s\mapsto P(\overline{\phi+s\psi})$ are the same as those for $\Pi(s)=\int_Y(e^{s\bar\psi}-1)\, d\mu_{\bar\phi}$, which can give us even more information on the form of $s\mapsto P({\phi+s\psi})$ when considering the proofs above.

 Clearly, under (H2), $\Pi(s)\to 0$, as $s\to 0$. Using standard arguments in probability theory(see, for instance,~\cite{Feller66}), one has
\begin{itemize}
\item[(a)] If (H2)(i) holds then $\Pi(s)=s(1+o(1))$.
\item[(b)] Suppose that (H2)(ii) holds. Then
\begin{itemize}
\item[(i)] if (H1)(b) holds then $\Pi(s)\le c \beta\Gamma(1-\beta) e^{sC}s^{-\gamma\beta}$.
This follows from the proof of Proposition~\ref{prop-deriveign} for estimating $\Psi(u)$ there.
\item[(ii)] if (H1)(a) holds then 
\begin{equation*}
C_2 \Gamma(1-\beta) e^{Cs}s^{\gamma\beta}(1+B(s))\le  \Pi(s)\le C_1 \Gamma(1-\beta) e^{Cs}s^{\gamma\beta}(1+B(s)),
\end{equation*}
where $B(s)\to 0$ as $s\to 0$. The above inequality follows by the argument used at the end of 
the proof of Proposition~\ref{prop-deriveign2} for estimating $\Psi(u)$ there .
\end{itemize}
\end{itemize}
\end{rmk}

\section{Limit properties of $\mu_{\phi+s\psi}$ as $s\to 0$}
\label{sec-LP}

We recall that in the set-up of Section~\ref{sec:abstsu} for fixed $s>0$, the potential $\phi+s\psi$ is 
positive recurrent. As clarified below, by allowing $s\to 0$, we move to the null  recurrent scenario.
Throughout this section, we continue to assume the set-up of Section~\ref{sec:abstsu}  and for simplicity, 
we further assume the following restriction on $f, F=f^\tau$ and $\bar\psi$:

\begin{itemize}
\item[\bf (H3)]  
\begin{itemize}
\item[\bf (i)] $\tau:Y\to\N$ is the first return time of $f$ to $Y$.
\item[\bf (ii)]  $\lim_{n\to\infty}\frac{\mu_{\overline{\phi + s\psi}}|_{\{\tau = n\}}}{e^{s\psi_n}\mu_{\bar\phi}|_{\{\tau = n\}}}=C'' > 0$ 
uniformly in $s$.
\item[\bf (iii)] (H2)(ii) holds for some $\gamma\in (0,1]$.
\end{itemize}
\end{itemize}
Assumption (H3)(ii)  (on $F$ and $\psi$) is quite strong but it is satisfied by maps in the family described in~\eqref{eq:nMPM} below.
In our examples in Section~\ref{sec:ex}, we start with an assumption on $\psi$ and verify the assumption below on $\bar\psi$.
Throughout this section (also later on in Section~\ref{sec-correl}) we will use a less restrictive form of (H1)(b), namely
\begin{itemize}
\item[\bf (H1)(b')] $\mu_{\bar\phi}(y\in Y:\tau(y)>n)=c n^{-\beta}(1+o(1))$ for 
some\footnote{Here we could allow that $\mu_{\bar\phi}(y\in Y:\tau(y)>n)=\ell(n) n^{-\beta}$, 
for a slowly varying function $\ell$, but for simplicity of notation we omit this.} $\beta\in(0,1)$.
\end{itemize}

We recall from Section~\ref{sec:abstsu}  that $F$ is a Markov map satisfying BIP and as such (P1) and (P2) hold
in the Banach space of bounded piecewise H{\"o}lder functions $\B\subset L^\infty(\mu_{\bar\phi})$ 
(see~\cite{AaronsonDenker01}). 
Throughout we set $\E_s(\tau) := \E_{\mu_{\overline{\phi + s\psi}}}(\tau)$. 
The next result below shows how $\E_s(\tau) \to \infty$ as $s\to 0$.
\begin{lemma}\label{lemma-tauexpec}
Assume (H3) and (H1)(b').
Then the following holds as as $s\to 0$,
\[
\E_s(\tau) \geq C_{\gamma,c} s^{(\beta-1)/\gamma}(1+o(1)),
\]
where $C_{\gamma,c}$ depends on $\gamma$ and the constant $c$ in (H1)(b').
\end{lemma}
As an immediate consequence of the above result and formula~\eqref{eq:pullback} (with $\mu_{\overline{\phi + s\psi}}$  instead of  $\mu_{\bar\phi}$), we have that  $\lim_{s\to 0} \mu_{\phi + s\psi}(X)=\infty$.
Since we also assume that $\tau$ is a first return, the limit measure $\lim_{s\to 0} \mu_{\phi + s\psi}$ 
is infinite and sigma-finite.

It seems natural to expect that $\mu_{\phi+s\psi}$ converges in some sense to $\mu_{\phi}$, as $s\to 0$. 
Since we are in an infinite measure setting with a nice first return inducing scheme, we approach  this
question via the `induced' measures. The next result gives estimates on how $\mu_{\overline{\phi+s\psi}}$ converges 
$\mu_{\bar\phi}$ as $s\to 0$ in a sense that is stronger than the weak$^*$ topology. 

\begin{lemma}\label{lemma-contmeas} 
Assume (H3). There exists $C>0$ such that for all continuous function  $v$ supported on $Y$ and for any $\eps<\beta$,
\[
|\mu_{\overline{\phi+s\psi}}(v)-\mu_{\bar\phi}(v)|\le C \sup_{y\in Y}|v(y)| s^{\beta-\eps}.
\]
\end{lemma}
The above statement together with~\eqref{eq:pullback} (with $\mu_{\overline{\phi + s\psi}}$  
instead of  $\mu_{\bar\phi}$)
implies that $|\mu_{\phi+s\psi}(v)-\mu_\phi(v)|\le C \sup_{y\in Y}|v(y)| s^{\beta-\eps}$
for any continuous function $v$ supported on $\cup_{i=1}^N(f^{-i}Y)$ for some fixed $N$.
But it doesn't give a result on the weak${}^*$ convergence of $\mu_{\phi+s\psi}$ 
for the entire space $C_0(X)$.

\subsection{Proof of Lemmas~\ref{lemma-tauexpec} and~\ref{lemma-contmeas}}

\begin{pfof}{Lemma~\ref{lemma-tauexpec}} Note that for $n \geq 1$, 
\begin{align*}
\E_s(\tau)&\geq \sum_{k=n}^\infty \mu_{\overline{\phi + s\psi}}(\tau>k)=\sum_{k=n}^\infty \sum_{j\ge k}\mu_{\overline{\phi + s\psi}}(\tau=j)
\end{align*}
Using the uniform convergence
in (H3)(ii), we can take $\eta \in (0,1/\gamma)$ and $n = [s^{-\eta}]$,
\[
\sum_{k=n}^\infty \sum_{j\ge k}\mu_{\overline{\phi + s\psi}}(\tau=j)
=C''\sum_{k=n}^\infty \sum_{j\ge k}e^{s\psi_j}\mu_{\bar\phi }(\tau=j)(1+o(1)) \quad \text{ as } s \to 0.
\]
By (H3)(iii) there exists $C', C>0$ such that
\begin{align*}
\sum_{j\ge k}e^{s\psi_j}\mu_{\bar\phi }(\tau=j)&\ge \sum_{j\ge k}e^{s(C'-Cj^\gamma)}\left(\mu_{\bar\phi }(\tau\ge j)-\mu_{\bar\phi }(\tau> j-1)\right)\\
&=e^{s(C'-Ck^\gamma)}\mu_{\bar\phi }(\tau\ge k)- e^{sC'} 
\sum_{j\ge k+1}\left(e^{-sC(j-1)^\gamma}-e^{-sCj^\gamma}\right)\mu_{\bar\phi }(\tau\ge j).
\end{align*}
Recall $\gamma\in (0,1]$. Then as $k\to\infty$ and $s\to 0$,
\begin{align*}
\sum_{j\ge k+1}\big(e^{-sC(j-1)^\gamma}&-e^{-sCj^\gamma}\big)\mu_{\bar\phi }(\tau\ge j)=\sum_{j\ge k+1}e^{-sCj^\gamma}\mu_{\bar\phi }(\tau\ge j)\left(e^{-sC((j-1)^\gamma-j^\gamma)}-1\right)\\
&\ll s\sum_{j\ge k+1}e^{-sCj^\gamma}\mu_{\bar\phi }(\tau\ge j)j^{-(1-\gamma)}\\
&\ll s\mu_{\bar\phi }(\tau> k)k^{-(1-\gamma)} \sum_{j\ge k+1}e^{-sCj^\gamma}\\
&\ll s\mu_{\bar\phi }(\tau> k)k^{-(1-\gamma)}e^{-sCk^\gamma}=o\left(e^{-sCk^\gamma}\mu_{\bar\phi }(\tau\ge k)\right).
\end{align*}
Hence, the following holds as $s\to 0$:
\[
\sum_{k=n}^\infty \sum_{j\ge k}\mu_{\overline{\phi + s\psi}}(\tau=j)
\sim e^{sC'}\sum_{k=n}^\infty e^{-sCk^\gamma}\mu_{\bar\phi }(\tau\ge k)(1+o(1)).
\]
For the right hand side, we compute that as $s\to 0$,
\begin{align*}
\sum_{k=n}^\infty e^{-sCk^\gamma}\mu_{\bar\phi }(\tau\ge k)&\sim \sum_{j \geq n} j^{-\beta} e^{-sC j^\gamma} 
  \sim \int_n^\infty x^{-\beta} e^{-sCx^\gamma} \ dx \\
  &= \gamma^{-1} (sC)^{\frac{\beta-1}{\gamma}} \int_{sCn^\gamma}^\infty u^{\frac{1-\beta}{\gamma} - 1} e^{-u} \ du\\
  &= \gamma^{-1} (sC)^{\frac{\beta-1}{\gamma}} 
  \left( \int_0^\infty - \int_0^{sCn^\gamma} \right) u^{\frac{1-\beta}{\gamma} - 1} e^{-u} \ du\\
  &\sim \gamma^{-1} (sC)^{\frac{\beta-1}{\gamma}}
  \left( \Gamma\left(\frac{1-\beta}{\gamma}\right) - \frac{\gamma}{1-\beta} (sCn^\gamma)^{\frac{1-\beta}{\gamma}} \right)\\
  &= s^{\frac{\beta-1}{\gamma}} 
\gamma^{-1} C^{\frac{\beta-1)}{\gamma}} \Gamma\left( \frac{1-\beta}{\gamma}\right) - 
\frac{n^{1-\beta}}{1-\beta}.
 \end{align*}
 The conclusion follows by putting all the above together and recalling that $n = [s^{-\delta}] = o(s^{-1/\gamma})$.
\end{pfof}

Before the proof of Lemma~\ref{lemma-contmeas}, we recall that the transfer operator $R$ associated with $F$ 
is defined w.r.t.\ the $F$-equilibrium measure $\mu_{\bar\phi}$.
Note that in the present set-up, the density $h=\frac{d\mu_{\bar\phi}}{dm_{\bar\phi}}$, where $m_{\bar\phi}$ is 
the $\bar\phi$-conformal measure,
is bounded and bounded away from zero. Moreover, we have the following pointwise formula for $R$:
\[
Rv(y)=\sum_{y\in F^{-1}x} e^{\bar\phi(y)} v(y)\frac{h(y)}{h(x)}.
\]
We recall that using the above formula one verifies that (P2) holds in the space $\B$ of bounded 
piecewise H{\"o}lder functions
by showing that a) the Lasota-Yorke inequality holds (see~\cite{AaronsonDenker01});
b) the space $\B$ is compactly embedded in $L^\infty(\mu_{\bar\phi})$ (see~\cite[Chapter 4]{Aaronson}).
Thus, given that $R^*$ is the dual of $R$ defined on $\B^*$ (the topological dual of $\B$), for any $\mu^*\in\B^*$
we have $(R^*)^n\mu^*\to \mu_{\bar\phi}$.

\begin{pfof}{Lemma~\ref{lemma-contmeas}} 
Let $R(s)v=R(e^{s\bar\psi}v)$. Under (P1) and (H3)(iii), it follows from the proof of Lemma~\ref{lemma-rus} that
there exists $\delta>0$ such that for all $s\in (0,\delta)$,
\[
\|R(s)-R(0)\|\ll s^{\beta-\eps}.
\]
Thus, there exists some  $\delta'\in (0,\delta)$ and a family of eigenvalues $\lambda(s)$
well defined on $(0,\delta')$. Moreover, the family $\lambda(s)$ is $C^{\beta-\eps}$ with $\lambda(0)=1$.
Since $\lambda(0)$ is isolated in the spectrum of $R$,
it follows that there exists $r>0$ such that for all  $s\in (0,\delta')$, the $r$-neighbourhood of $\lambda(s)$ is
disjoint from the rest of the spectrum of $R(s)$. 
As a consequence, given that $R(s)^*$ acting on $\B^*$ is the dual operator of $R(s)$,
we have that $R^*(s) \mu_{\overline{\phi+s\psi}}=\lambda(s) \mu_{\overline{\phi+s\psi}}$
and $\lambda(s)^{-n}(R(s)^*\mu^*)^n\to \mu_{\overline{\phi+s\psi}}$, for any $\mu^*\in\B^*$.

By an argument similar to the one used in the proof of Lemma~\ref{lemma-rus}, 
there exists $\delta>0$ such that for all $s\in (0,\delta)$,
\[
\|R(s)^*-R(0)^*\|\ll s^{\beta-\eps}.
\]
It follows from standard perturbation theory (see~\cite{Kato}) that the same continuity property  holds for the family
of eigenmeasures $\mu_{\overline{\phi+s\psi}}$. Hence, for any $v\in\B$,
$|\mu_{\overline{\phi+s\psi}}(v)-\mu_{\bar\phi}(v)|\ll |v|_\infty s^{\beta-\eps}$.
Because $\B$ is dense in $C_0(Y)$, the conclusion follows.~\end{pfof}

\section{Statistical properties under (H3)}
\label{sec-correl}

Throughout this section, we assume the assumptions of Section~\ref{ssec:basic} along with (H3) and (H1)(b') in Section~\ref{sec-LP}. Sometimes we will need a few stronger assumptions as specified below.
We recall from Section~\ref{sec-LP} (see Lemma~\ref{lemma-tauexpec} and the explanation after the statement) 
that under these assumptions $\lim_{s\to 0}\mu_{\phi+s\psi}(X)=\infty$
and thus the limit of the sequence of potentials $\phi+s\psi$, as $s\to 0$ is null recurrent.

\subsection{Correlation coefficients as $s\to 0$}
We are interested in the asymptotics of the correlation function
\[
\rho_{n,s}(v,w)=\int_X v w\circ f^n\, d\mu_{\phi+s\psi},
\]
as $n\to\infty$ and $s\to 0$. In what follows we show that the main results in~\cite{Gouezel11, MT12} apply 
in this null recurrent potential setting.

For $\theta\in (-\pi,\pi]$ we define the complex version of $R(u)$ by
\[
R(\theta)v:= R_{\mu_{\bar\phi}}(e^{i\theta\tau}v).
\]
Let $R_n=R(1_{\{\tau=n\}})$ as in Section~\ref{sec:abstsu}. 
To recall all important results (including~\cite{Gouezel11}), we shall also
need the following stronger version of (P1), which holds  (in the Banach space $\mathcal{B}$)  under the assumptions of Section~\ref{ssec:basic} and  (H1)(b') formulated in Section~\ref{sec-LP}. 
\begin{itemize} 
\item[\bf (P1')] $\|R_n\|\ll \mu_{\bar\phi}(\tau=n)\ll n^{-(\beta+1)}$ with  $\beta$ as in (H1)(b').
\end{itemize}
For the verification of bf (P1') we refer to~\cite{Sarig02}.

 It is known that a certain aperiodicity assumption is required to obtain the exact asymptotics for the 
 correlation function (see~\cite{Sarig02, Gouezel04, Gouezel11, MT12}).
In the present set-up we shall need the following  stronger version of (P2), which holds (in the Banach space $\mathcal{B}$) under the assumptions of Section~\ref{ssec:basic}.
\begin{itemize}
\item[\bf (P2')] 
\begin{itemize}
\item[(i)] The eigenvalue $1$ is simple and isolated in the spectrum of $R(0)$.
\item[(ii)] For $e^{i\theta}\ne 1$, the spectrum of $R(\theta)$ does
not contain $1$.
\end{itemize}
\end{itemize}
The main result of this section states that the correlation coefficients
of $\mu_{\phi+s\psi}$ behave as in the null recurrent case, not only when $s = 0$, but also as
$s \to 0$ sufficiently quickly compared to the rate at which $n\to\infty$.

\begin{thm}\label{thm-mix}
 Assume (H1)(b') and (H3) and recall that $(P1')$ and  (P2') hold. 
Let $v, w$ be functions supported on $Y$  with $v\in\B$ and $w$ bounded. Then
\begin{equation}
\label{eq-mix}
\lim_{\stackrel{s\to 0}{n\to\infty}}  n^{1-\beta}\int_Y v w\circ f^n \, d\mu_{\phi+s\psi}=\int_Y v \, d\mu_{\phi}\int_Y w \, d\mu_{\phi},
\end{equation}
where the simultaneous limit is taken such that $s=o(n^{-(1-\beta)/(\beta-\eps)})$, for $\eps>0$ arbitrarily small.
\end{thm}

\begin{rmk}\label{rmk:7.2}
Using arguments as in~\cite{BT16}, one can deal with larger classes of 
functions and the case when $\tau$ is a general return time, but this goes beyond the scope of this paper.
\end{rmk}

We recall that the asymptotics of the correlation function for infinite measure preserving systems can be obtained 
via the use of operator renewal theory as in~\cite{Gouezel11, MT12}
(introduced in the context of dynamical system in~\cite{Sarig02} to be obtain lower bounds for the correlation 
decay for finite measure preserving systems).
Put $T_nv=1_Y L_\phi^n(1_Yv)$ and recall that the relationship $T_n=\sum_{j=1}^n T_{n-j}R_j$
generalises the notion of scalar renewal sequences (see~\cite{Feller66, BGT} and references therein).
Consider the following conditions:
\begin{itemize}
\item[(a)] $\beta\in (1/2, 1)$ and (P1);
\item[(b)] $\beta\in (0,1)$ and (P1').
\end{itemize}

Using (a), Theorem~\ref{thm-mix},~\cite[Theorem 1.1]{MT12}
establishes that for all $v\in\B$,
\begin{equation}
\label{eq-Tn}
\lim_{n\to\infty} n^{1-\beta}T_n v=C_\beta\int_Y v\, d\mu_{\bar\phi},
\end{equation}
uniformly on $Y$, where $C_\beta$ is a constant that depends on $\beta$ and the value of 
the density $h=\frac{d\mu_{\bar\phi}}{d m_{\bar\phi}}$ at some point $y\in Y$, where $m_{\bar\phi}$ is the 
$\bar\phi$-conformal measure for the induced map $F$.
Equally importantly,~\cite[Theorem 1.1]{Gouezel11}  shows that~\eqref{eq-Tn} holds under (b).

In what follows we use~\eqref{eq-Tn} to obtain the asymptotics of $\int_Y v w\circ f^n \, d\mu_{\phi+s\psi}$. 
Since we want to let $s\to 0$ we will also use  Lemma~\ref{lemma-contmeas}.

\begin{pfof}{Theorem~\ref{thm-mix}} Since $v, w$ are supported on $Y$, we write
\begin{align*}
\int_Y v w\circ f^n \, d\mu_{\phi+s\psi}=\int_Y v w\circ f^n \, d\mu_{\overline{\phi+s\psi}}
=\int_Y v w\circ f^n \,d\mu_{\bar\phi} +K_{v,w}(s),
\end{align*}
where $K_{v,w}(s)=\int_Y v w\circ f^n \, (d\mu_{\overline{\phi+s\psi}}-d\mu_{\bar\phi})$.
Recall $v\in\B$ and $w$ is a bounded function supported on $Y$. Using Lemma~\ref{lemma-contmeas},
\begin{align*}
|K_{v,w}(s)|\ll\|v\|_{\B}\sup_{y\in Y}|w(y)|\Big|\int_Y  (d\mu_{\overline{\phi+s\psi}}-d\mu_{\bar\phi})\Big|\ll \|v\|_{\B}\sup_{y\in Y}|w(y)| s^{\beta-\eps}.
\end{align*}
 By~\eqref{eq-Tn},
\[
\lim_{n\to\infty}n^{1-\beta}\int_Y v w\circ f^n \,d\mu_{\bar\phi}\to C_\beta \int_Y v\, \,d\mu_{\bar\phi}\int_Y w\, \,d\mu_{\bar\phi}.
\]
By assumption, $s=o(n^{-(1-\beta)/(\beta-\eps)})$, so $s^{\beta-\eps}=o(n^{-(1-\beta)})$ and the conclusion follows.
\end{pfof}

\subsection{Arcsine law as $s\to 0$}

In this section we briefly spell out some consequences of
Lemma~\ref{lemma-contmeas} for limit laws in this null recurrent limiting
potential setting.

We will be interested in an arcsine law (studied for dynamical systems
with infinite measure by Thaler, in for instance,~\cite{Thaler00}, see
also references therein) w.r.t.\ the limit measure $\lim_{s\to 0}
\mu_{\overline{\phi+s\psi}}$.

For $n\ge1$ and $x\in \bigcup_{j=0}^n f^{-j}Y$, let
\[
Z_n(x)=\max\{0\leq j\leq n:f^jx\in Y\}
\]
denote the time of the last visit of the orbit of $x$ to
$Y$ during the time interval $[0,n]$.

Let $\zeta_{\beta}$ denote a random variable distributed according to
the $\mathcal{B}(1-\beta,\beta)$ distribution:
\[
\P(\zeta_\beta\leq t)=
\frac{1}{\pi}\sin\beta\pi  \int_0^t\frac{1}{u^{1-\beta}}\frac{1}{(1-u)^{\beta}}\,du,
\quad t\in [0,1].
\]

\begin{prop}\label{prop-arcsine}
Assume the set-up of Theorem~\ref{thm-mix} with either (a) or (b).
Let $\nu_s$ be an absolutely continuous probability measure on $Y$ with
density $g_s\in\mathcal{B}$. Moreover, suppose that there exists
$g\in\B$ such that as $s\to 0$, $\|g_s-g\|\ll s^{\beta-\eps}$, for any $\eps<\beta$.

Then
\[
\lim_{\stackrel{s\to 0}{n\to\infty}}\nu_s\left\{\frac1n Z_n\leq t\right\}=\P(\zeta_\beta\leq t),
\]
where the simultaneous limit is taken such that $s=o(n^{-(1-\beta)/(\beta-\eps)})$, for $\eps>0$ arbitrarily small.
\end{prop}

\begin{proof} The proof goes word for word as~\cite[Corollary
2]{Thaler00} (see also~\cite[Corollary 9.10]{MT12},\cite[Corollary
3.5]{T15} for error rates)
with the additional use of $\|g_s-g\|\ll s^{\beta-\eps}$ and
Lemma~\ref{lemma-contmeas}.~\end{proof}

Versions of the Darling-Kac law, where the involved distributional
convergence is with respect to any limit measure $\lim_{s\to 0} \nu_s$ as in the
statement of Proposition~\ref{prop-arcsine}, can be obtained using the
classical results in~\cite{Aaronson} and the present
Lemma~\ref{lemma-contmeas}.

\section{Examples}\label{sec:ex}

We first note that in all of our examples, \eqref{eq-reind} holds since either we are dealing with first return maps, or our reinducing time is a first return map on a finite measure set.

\subsection{Checking the liftability condition}

When the induced system has the BIP property, many of the conditions we assume in the abstract 
set-up hold automatically. Indeed, under the full branches condition, 
any sufficiently smooth ($SV$) induced potential has an RPF measure which satisfies the Gibbs property.
The next lemma, a version of \cite[Lemma 4.1]{IomTod10}, goes a little way to ensuring the 
liftability condition required in (H2). We write $\uu$ for the potential; this can be $\phi$ or $\phi + s\psi$
as the application requires.

\begin{lemma}\label{lem:induced p nonpos}
Suppose that $(Y, F)$ has the BIP property, and $\bar\uu\in SV$. 
If $P(\uu)<\infty$ then $P(\overline{\uu-P(\uu)})\le 0$.
\end{lemma}

We next give a setting where we can prove also 
$P(\overline{\uu-P(\uu)})\ge 0$ which will be relevant for our examples, 
and gives us a criterion to check liftability. Note that if we know we have an equilibrium state 
which lifts to $(Y, F)$ the following holds immediately.

\begin{lemma} \label{lem:lift_shadow}
Suppose that $(Y, F)$ has the BIP property and $P(\uu)\in (0, \infty)$, 
$\bar\uu\in SV$ and $P(\bar\uu)<\infty$.  
Suppose further that for each $\eps>0$ there is an ergodic $\mu\in \M$ which lifts 
to $(Y, F)$ with $h_\mu(f)+\int\uu~d\mu> P(\uu)-\eps$.  Then 
$P(\overline{\uu-P(\uu)})=0$.
\end{lemma}

\begin{proof}
We aim to show that $P(\overline{\uu-P(\uu)})= 0$.   Let $\bar\mu$ be the RPF measure for $\overline{\uu}$.  
First we claim $P(\bar\uu)> 0$.
Indeed, since $P(\uu)> 0$, there is a measure $\nu$ with $h(\nu)+\int\uu~d\nu>0$, which, 
by Abramov's formula, induces to a measure $\bar\nu$ with $h(\bar\nu)+\int\bar\uu~d\bar\mu>0$. This proves the claim.

Since $\tau\ge 1$, we must have
$P(\overline{\uu-p})\le P(\overline{\uu})-p<\infty$ for all $p>0$ and $p\mapsto P(\overline{\uu-p})$ is a 
continuous decreasing function tending to $-\infty$ as $p\to \infty$. 
Hence, also using the claim and Lemma~\ref{lem:induced p nonpos}, there exists $p>0$ such that $$P(\overline{\uu-p})=0.$$
So the lemma follows if we can show that $p= P(\uu)$.  Let $\bar\mu'$ be the RPF measure 
for the potential  $\overline{\uu-p}$. 
Let $\{ Y_i \}_{i \geq 1}$ be the Markov partition of the induced map.
The Gibbs property applied to $\bar\mu'$ and $\bar\mu$ respectively gives
\begin{eqnarray*}
\int\tau~d\bar\mu' &\asymp& \sum_nn \bar\mu'(\tau=n)
\asymp \sum_n \sum_{\tau(Y_i)=n} e^{S_n\uu-np} \\
&=& \sum_n n e^{-np} \sum_{\tau(Y_i)=n} e^{S_n\uu}
\asymp \sum_n n e^{-np}  \bar\mu(\tau=n) <\infty.
\end{eqnarray*}
This projects to a measure $\mu'$ with
$$h_{\mu'}+\int \uu-p~d\mu=\frac1{\int\tau~d\bar\mu'}\left(h_{\bar\mu'}+\int \overline{\uu-p}~d\bar\mu'
\right)=0.$$
(Note that this also shows that $\mu'$ is an equilibrium measure.)
So $p\le P(\uu)$.  If $p< P(\uu)$ then there is a measure $\nu$ with $h_\nu+\int \phi+s\psi-p~d\nu>0$.  
So $\nu$ induces to a measure $\bar\nu$ which by Abramov's formula has $h_{\bar\nu}+\int \overline{\uu-p}~d\bar\nu>0$, a contradiction.  
Hence $p= P(\uu)$, proving the lemma. 
\end{proof}

\subsection{Checking $P(\phi + s\psi) = o(s)$}

In Theorems~\ref{thm:PbarP} and \ref{thm:PbarP2} we require that
$s\ll (P(\phi+s\psi))^a$ for some $a>0$.  The following lemma guarantees this in many settings.

\begin{lemma}
Suppose that $f:X \to X$ has a BIP induced map $F = f^\tau:Y \to Y$, such that for $\phi$ 
as in Section~\ref{sec:psi}, $\mu_{\bar\phi}(\tau=n) \lesssim n^{-(1+\beta)}$
with $\beta \in (0,1)$. 
Suppose that $\phi$ is a potential with $P(\phi+s\psi)>0$ for $s>0$, $\bar\psi\in SV$
with $\int \bar\psi \, d\mu_{\bar\phi} > 0$, and there is $C>0$ and $N\in \N$ 
such that $\bar\psi|_{\{ \tau = n \}} \leq C \log n$ for $n\ge N$.  
Then for any $\eps>0$, there is $C_0>0$ such that for all $s \ge 0$, 
$P(\phi+s\psi)\ge  C_0s^{1/(\beta-\eps)}$.
\end{lemma}

\begin{proof}
We label the domains of $F$ by $Y_i$ and for each $i$, let $x_i$ be some point in $Y_i$.
Write $p(s) = P(\phi+s\psi)$.
For $\delta \in [0,1)$, let $\mu_s^\delta$ be the RPF measure of $\overline{\phi+s\psi-\delta p(s)}$.
This yields 
\begin{align*}
 \int_Y \tau \, d\mu_s^\delta &\asymp \sum_{n\ge 1}n \sum_{\tau(Y_i)=n}  
 e^{\overline{\phi+s\psi-\delta p(s)}(x_i)}
  \lesssim  \sum_{n=1}^\infty n\mu_{\bar\phi}(\tau=n)
  \sum_{\tau(Y_i)=n} e^{\overline{s\psi-\delta p(s)}(x_i)} \\
 &\lesssim  \sum_{n=1}^\infty \frac{n}{n^{1+\beta}} n^{sC} e^{-n\delta p(s)} 
  \asymp \int_1^\infty x^{-\beta+sC} e^{-x\delta p(s)} \, dx \\
 &= (\delta p(s))^{\beta-1-sC } \int_{\delta p(s)}^\infty y^{-\beta+sC } e^{-y} \, dy 
   \sim \delta^{\beta-1-sC}\, \Gamma(1-\beta+sC)\, p(s)^{\beta-1-sC}
\end{align*}
as $s\to 0$. 
In particular, $\int \tau \, d\mu^\delta_s < \infty$, so by \cite[Theorem 2]{Sarig01a},
$\mu^\delta_s$ is an equilibrium state. Moreover, $\mu^\delta_s$ projects 
to an $f$-invariant measure $\nu^\delta_s$.
By the definition of pressure and Abramov's formula,
\begin{eqnarray*}
(1-\delta) p(s) &\geq& 
\left( h(\nu_s^\delta) + \int \phi + s\psi\, d\nu_s^\delta\right) - \delta p(s)  =
h(\nu_s^\delta) + \int \phi + s\psi-\delta p(s) \, d\nu_s^\delta \\
&=& \frac{1}{\int \tau \, d\mu_s^\delta} 
\left(h(\mu_s^\delta) + \int \overline{\phi+s\psi-\delta p(s)} \, d\mu_s^\delta \right)
= \frac{P(\overline{\phi+s\psi - \delta p(s)})}{\int \tau \, d\mu_s^\delta}.
\end{eqnarray*}
We assumed that $\int\bar\psi~d\mu_{\bar\phi}=K > 0$, so by smoothness and convexity of the 
pressure function,  $P(\overline{\phi+s\psi}) = Ks+ h.o.t.$.
By the analyticity of 
$(s, z) \mapsto P(\overline{\phi+s\psi-z})$ on $(0, \infty)\times(0, \infty)$ 
(cf.\ \cite[Theorem 6.5]{Sarig01a}), we can take $\delta > 0$ sufficiently small so that
$P(\overline{\phi+s\psi-\delta p(s)}) \geq Ks/2$ for all small $s \geq 0$.
Combining the above and making $p(s)$ subject of the inequality, we get
$$
p(s) \gg s^{1/(\beta-sC)} \geq s^{1/(\beta-\eps)} \qquad \text{ as } s \to 0.
$$
\end{proof}

\begin{rmk}
In order to apply this to the examples later in this section, we can verify the assumption
$p(s) := P(\phi+s\psi) > 0$ for $s > 0$ using a continuity argument.
First assure that $K := \int \bar\psi \, d\mu_s > 0$; in the case where $\psi$ satisfies (H2)(ii),
this boils down to taking $C'$ sufficiently large.
Then truncate the induced system to states with $\tau(Y_i) \leq N$, then the corresponding
pressure $P_N(\overline{\phi+s\psi}) \geq Ks/2$ for a sufficiently large $N$.
At the same time, for the RPF measure $\mu_{s,N}$, we have
$$
\int_Y \tau\, d\mu_{s,N} \leq \sum_{n\leq N} \sum_{\tau(Y_i)=n} e^{\overline{s\psi-\delta p(s)}(x_i)} 
\lesssim \sum_{n \leq N} \frac{n^\gamma}{n^{1+\beta}} < \infty,
$$
whence by Abramov's formula $p(s) \geq p_{s,N} \geq Ks/(2\int \tau \, d\mu_{s,N}) > 0$.
\end{rmk}

\subsection{AFN maps}
\label{subsec:afn}

The main class of maps where the previous results apply are AFN maps, i.e., non-uniformly
expanding interval maps with finitely many branches, finitely many neutral fixed points,
and satisfying Adler's distortion property ($f''/{f'}^2$ bounded).
The two-parameter family of such maps is the (non-Markov) Pomeau-Manneville 
maps $f:[0,1] \to [0,1]$ defined as
\begin{equation}\label{eq:nMPM}
 f(x)=f_{\alpha, b}(x) = \begin{cases}
    x(1 + 2^\alpha x^\alpha) & \text{ if } x \in [0,\frac12];\\
    b(2x-1) & \text{ if } x \in (0,1],
        \end{cases}
 \qquad \alpha > 1, \ b \in (0,1].
\end{equation}
The right branch of $f$ is not expanding if $b \in (0, \frac12]$, but 
the first return induced map $F$ to $Y = (\frac12,1]$ is uniformly expanding
for all $b \in (0,1]$, so we can allow all $b \in (0,1]$.
Still $F$ need not be Markov, but as shown in \cite[Section 9]{BT16}, 
one can construct a reinduced 
Gibbs-Markov map to $Y$ that satisfies the hypotheses in the previous section as we will see.

For the standard Pomeau-Manneville map (i.e., $b=1$) with potential 
$\phi_t = -t\log f'$, the shape of the pressure
was computed by Lopes \cite[Theorem 3]{L93}:
$$
P(\phi_t) \asymp \begin{cases}
h_\mu(1-t) + B(1-t)^{1/\alpha} + \text{ h.o.t.} & \text{ if } t < 1 \text{ and } \alpha \in (\frac12,1);\\
 C(1-t)^\alpha + \text{ h.o.t.} & \text{ if } t < 1 \text{ and } \alpha > 1;\\
0 & \text{ if } t \ge 1,
\end{cases}
$$
where $h_\mu$ is the entropy of the non-Dirac equilibrium measure
 (i.e., the acip) and $B, C > 0$ are constants.
 Note that the transition case $\alpha=1$ corresponds to the transition
 from a finite acip (for $\alpha < 1$) to an infinite acim (for $\alpha \ge 1$).
The next proposition makes the case $\alpha > 1$ more precise.  Note that Lopes' result in the $\alpha>1$ case follows from the proposition when one notices that $t\mapsto  P(\overline{\phi_t})$ is of the form $K(1-t) +O((1-t)^2)$ where $K=\int\bar\phi~d\mu_{\bar\phi}$.

\begin{prop}\label{prop:psilogtau}
Let $f=f_{\alpha, b}$ be as in \eqref{eq:nMPM}
and take $\phi = -\log f'$. Then $\bar \phi \in SV$.  
\begin{itemize}
 \item [(a)]
Suppose that $\psi(x) = -x^\eta$ for $x$ in a neighbourhood of $0$ with $\eta = (1-\gamma)\alpha$
and $\gamma\in (0, \alpha/(1+\alpha))$. Moreover, assume that $P(\phi+s\psi)>0$ for $s > 0$.  Then
$$
P(\phi+s\psi) = (CP(\overline{\phi+s\psi}))^{\alpha}(1+Q(s))
$$
as $s \to 0$ for $C$ and $Q$ as in Theorem~\ref{thm:PbarP} with $\alpha = 1/\beta$. 
\item[(b)]
If $\phi_t =- t \log f'$, then 
$$
P(\phi_t) = (C P(\overline{\phi_t}))^{\alpha}(1+Q(1-t))
$$ 
as $t\nearrow 1$ for $C$ and $Q$ as in Theorem~\ref{thm:PbarP} (i.e., for $s=1-t$). 
\item[(c)] For $b=1$, the conclusions of Theorem~\ref{thm-mix} and Proposition~\ref{prop-arcsine}
hold.
\end{itemize}
\end{prop}

\begin{rmk}
 The assumption $b=1$ in part (c) is used to ensure that a Gibbs-Markov {\em first} return map $F$ is available.
 Without it, extra arguments are needed which we omit for simplicity, see Remark~\ref{rmk:7.2}.
\end{rmk}

\begin{proof}
The maps $f$ with $b=1$ fit in the class of studied in~\cite{LiveraniSaussolVaienti99} 
with $\phi = -\log f'$ and satisfies (H1)(b) (see~\cite[Proposition A3]{T15}).
For general $b \in (0,1]$, \cite[Section 9]{BT16} gives 
$\mu_{\bar\phi}(\tau > n) \sim c n^{-\beta}$ for $\beta = 1/\alpha$
and condition (H1)(b) holds.
Here $c = c^* h(\frac12)$, where $h$ is the $F$-invariant density and $c^*$ is given in 
\cite[Section B]{T15}.
The induced map is Gibbs-Markov, and together with (H1)(b) this assures that (P1) and (P2) hold.

The induced potential $\bar\phi = -\log|F'|$ and for each $n$-cylinder $C_n$, we have
\begin{equation}\label{eq:var}
\sup_{x,y \in C_n} |\bar\phi(x)-\bar\phi(y)| \leq \int_{C_n} |(-\log |F'(\xi)|)'| \ d\xi
= \int_{C_n} \frac{|F''(\xi)|}{|F'(\xi)|^2} |F'(\xi)| \ d\xi
\leq A |F(C_n)|,
\end{equation}
where $A := \sup |F''(\xi)|/|F'(\xi)|^2 < \infty$ by Adler's condition, cf.\ \cite[Lemma 10]{Z98}.
Since $F$ is uniformly expanding, $\sup\{|C_n| : n\text{-cylinders}\}$ decreases exponentially in $n$,
so $\bar\phi\in SV$.

To prove (a), the potential $\psi(x) = -x^\eta$ for $\eta = (1-\gamma)\alpha$
(with $\gamma \in (0,1]$) and $x$ close to $0$
produces an induced potential $\bar\psi(x) \sim C-\tau(x)^\gamma$ as $x \to \frac12$.

To check (H2), first note that the H\"older regularity guarantees that $\bar\psi\in SV$ (since $F$ is uniformly expanding, 
it suffices to bound the first variation: for $x, y$  in the same $1$-cylinder $Y_i$ we have
$|\bar\psi(x)-\bar\psi(y)|\lesssim \sum_{k=0}^{\tau(x)-1}|f^k(Y_i)|^\gamma
\lesssim \sum_{n=1}^\infty n^{-\gamma(1+\frac1\alpha)}<\infty$).  To conclude that (H2) holds, 
 we need to check liftability. 
 First notice that there is only one measure in $\M$, namely $\delta_0$, which does not lift to $(Y, F)$, 
 so for a given potential $\uu:[0,1]\to [-\infty, \infty)$, 
Lemma~\ref{lem:lift_shadow} follows if $\delta_0$ is ``not isolated'': there is a sequence $(\mu_n)_n$ of ergodic measures in $\M\setminus\{ \mu\}$
such that $h_{\mu_n}+\int\uu~d\mu_n\to h_{\mu}+\int\uu~d\mu$.  
For this, one can take a sequence of Dirac masses on 
periodic cycles with this property, so liftability follows by Lemma~\ref{lem:lift_shadow}.

With these hypotheses checked, Theorem~\ref{thm:PbarP} gives the result.

For statement (b) on the potential $\phi_t$, we take $\phi = -\log|f'|$, $\psi = \log|f'|$ and $s=1-t$,
so $\bar\psi = s\log|F'| \in L^1(\mu_{\bar\phi})$ and (H2)(i) holds by Remark~\ref{rmk-log}. 
Properties (P1) and (P2) and (H2) are shown as above.
It is standard that $P(-t\log|f'|)>0$ if $t<1$.
Hence Theorem~\ref{thm:PbarP} applies with $C = (c \beta \Gamma(1-\beta))^{-1}$.

For (c), the verification of (H1)(b'), (P1') and (P2') are similar to the verification of (H1)(b),
(P1) and (P2). 
To check (H3)(ii),
take $s \geq 0$ and let $h_s := \frac{d\mu_{\overline{\phi+s\psi}}}{dm_{\overline{\phi+s\psi}}}$
be the density of the $F$-invariant equilibrium state w.r.t.\ 
the $\overline{\phi+s\psi}$-conformal
measure on $Y = [\frac12,1]$; they are H\"older because $\psi$ is.
Since $F$ has full branches, we compute as $n \to \infty$:
\begin{eqnarray*}
\mu_{\overline{\phi+s\psi}}(\tau=n) &\sim& h_s\left(\frac12\right) m_{\overline{\phi+s\psi}}(\tau=n)
\sim h_s\left(\frac12\right) e^{S_n(\phi+s\psi)} \\
&=& h_s\left(\frac12\right) e^{s \, S_n\psi} m_{\overline{\phi}}(\tau=n) \sim
\frac{h_s(\frac12)}{h_0\left(\frac12\right)}  e^{s \, S_n\psi} \mu_{\overline{\phi}}(\tau=n).
\end{eqnarray*}
The continuity given in Lemma~\ref{lemma-rus} carries over to the density, so
that $\lim_{s\to 0} \frac{h_s(\frac12)}{h_0(\frac12)} = 1$.
This verifies (H3)(ii).
Therefore Theorem~\ref{thm-mix} and Proposition~\ref{prop-arcsine}
apply.
\end{proof}

\subsection{Unimodal maps with flat critical points}\label{sec:flat}

Another source of infinite measure systems on the interval are $C^2$ Misiurewicz maps
with a flat critical point. The study of these goes back to \cite{BM89},
with further contributions by e.g.\ \cite{Th05, Z04, T17}.
The map $f:[-1,1] \to [-1,1]$ has a {\em flat} critical point $0$ 
which means all its derivatives vanish.
For simplicity we let the Misiurewicz condition (non-recurrence of the critical point)
mean here that the critical orbit is $0 \mapsto 1 \mapsto -1 \circlearrowleft$.
In addition, we assume that $f$ is smooth and symmetric (i.e., $f(x) = f(-x)$).
In this case, we can look at the first return map
$F=f^\tau:[-p,p] \to [-p,p]$ where $p > 0$ is the orientation reversing fixed point of $f$.
This is a Gibbs-Markov map, with invariant measure $\mu \ll \Leb$
and density $h = \frac{d\mu}{d\Leb}$ is smooth, bounded and bounded away from zero. 
Since all points except preimages of $0$ enter $[-p,p]$ under iteration, the only potentials 
which are not liftable are those with equilibrium measure $\delta_0$, so as in the Pomeau-Manneville case, we only need 
to check if $\delta_0$ is isolated for our potential of interest.  
It is easy to see that for $\phi=\log|f'|=-\psi$ this holds.

In \cite{T17}, it is shown that for the potential $\phi_t = -t \log|f'|$, there is a phase transition
at some $t^+>0$ and $P(\phi_t) = 0$ for $t \geq t^+$, but
a proof that $t^+ = 1$ (as one would conjecture) is not given, and the order of the phase 
transition, and the shape of the pressure for $t < 1$ are not computed. 

\begin{lemma}
If $t<1$ then $P(-t\log|f'|)>0$, i.e., $t^+=1$.
\end{lemma}

\begin{proof}
We first note that the potential $-\log|f'|$ is recurrent. One way is to see this is that every point in 
$[-1,1]\setminus \cup_{n\ge 0}f^{-n}(-1)$ returns to $[-p, p]$ infinitely often.  
This is a full Lebesgue measure set of points, so Lebesgue is conservative. 
Whence $P(-\log|F'|)=0$.  The RPF measure $\bar\mu$ can be shown to have finite, strictly positive, 
Lyapunov exponent $\int\log|F'|~d\bar\mu$.  This can be seen from the Gibbs property of $\bar\mu$ and 
the tail estimates of \cite[Theorem 5]{Z04}.  Also by those tail estimates, there is $\eps>0$ such 
that $P(-t\log|F'|)<\infty$ for $t\in (1-\eps, \infty)$.  Combining these two facts, we see 
that $t\mapsto P(-t\log|F'|)$ is analytic on $(1-\eps, \infty)$ with derivative 
$-\int\log|F'|~d\bar\mu<0$ at $t=1$, so strictly decreasing.  
Since by Lemma~\ref{lem:induced p nonpos}, $P(-t\log|F'|-\tau P(-t\log|f'|)\le 0$, 
this gives $P(-t\log|f'|)>0$ for $t<1$.
\end{proof}

Combining \cite[Formula (6)]{Z04} with the computation of \eqref{eq:var}, we obtain that $\bar\phi \in SV$.
The tail estimates in \cite[Theorem 5]{Z04} show regular variation, but
since they are general and don't include the higher order terms\footnote{In \cite{T17}, 
there are tail estimates (Proposition 2.6 part 3), but they are a bit imprecise
and due to the form of map, there is an extra logarithmic term which is not present in
the setting presented here.}, we present a slightly different family, and sketch
the argument for the tails. Set 
\begin{equation}\label{eq:flat}
f(x) = \begin{cases}
  1-2 e^{-b(|x|^{-\alpha} - 1)} & \text{ if } x \neq 0,\\
  1 & \text{ if } x = 0,
       \end{cases}
\qquad\text{ for } b, \alpha > 0,
\end{equation}
and let
$$
V_{\pm}(y) = \pm \left( 1 - \frac1b \log \frac{1-y}{2} \right)^{-\beta} \qquad \text{ for } \beta = 1/\alpha,
$$
be the two inverse branches of $f$.
Then the region $\{ \tau > n \} = (-a_n, a_n)$ where
$\pm a_n = V_{\pm} \circ V_+ \circ V_-^{n-2}(p)$.
Let $p_n = f(\pm a_n) = V_+ \circ V_-^{n-2}(p) = 1-\zeta_n r^n$, where 
$r = |f'(1)|^{-1} = |f'(-1)|^{-1}$.
Then $\zeta_{n+1} = 1-V_+(\zeta_n r^n) = \zeta_n r^{n+1} (1+ O(r^{n-1} \zeta_n^2))$,
so $(\zeta_n)_{n \geq 1}$ converge exponentially fast to some positive limit $\zeta$.
This gives that for some $c > 0$,
\begin{eqnarray*}
 \pm a_n &=& V_\pm(\zeta r^{n-1}(1 + O(e^{-cn}))) \\ 
 &=& \pm n^{-\beta} \left( \frac{b}{\log |f'(1)|}\right)^{\beta}
 \left(1 - \frac{\beta}{n} \frac{\log \zeta/2}{\log |f'(1)|} + O(n^{-2}) \right).
\end{eqnarray*}
Now to estimate $\mu(\tau > n)$ we integrate over the density $h(x) = h_0 + h_1 x + h_2 x^2 + O(x^3)$.
This shows that
$$
\mu(\tau > n) = \int_{-a_n}^{a_n} h_0 + h_1 x + h_2 x^2 + O(x^3) \, dx
= 2h_0 a_n + \frac{2}{3} h_2 a_n^3 + O(a_n^4)
$$ 
satisfies condition (H1)(b) in Section~\ref{sec:ev} when $\beta \in (\frac23,1)$. 
(For smaller values of $\beta$, we just introduce more terms in the expansion of the density $h$.)
Therefore Theorem~\ref{thm:PbarP} applies to the family \eqref{eq:flat}
and potential $\phi = -\log|f'|$, so we obtain:

\begin{prop}\label{prop:psilogtau2}
Assume $\psi$ is as in Proposition~\ref{prop:psilogtau} (here a neighbourhood around $-1$ takes the 
role of the neighbourhood around 0) with $\alpha > 1$.  Then
$$
P(\phi+s\psi) = (CP(\overline{\phi+s\psi}))^{\alpha}(1+Q(s))
$$
as $s \to 0$ for $C$ and $Q$ as in Theorem~\ref{thm:PbarP}. 

In particular if $\phi_t = -t \log|f'|$, then there is $C$ such that
$P(\phi_t) = (C P(\overline{\phi_t}))^{\alpha}(1+Q(1-t))$ as $t\nearrow 1$ for $C$ and $Q$ as in 
Theorem~\ref{thm:PbarP} (i.e., for $s=1-t$). 
\end{prop}

\begin{rmk}
If $\psi(x) = \eta(-\log(x+1))^{\gamma-1}$ for some $\eta > 0$, $\gamma \in (0,1]$ and all $x$ close to $-1$,
then  $\bar\psi(y) \sim C - \frac{\eta}{\gamma} (-\log r)^{1-\gamma} \tau(y)^\gamma$ for $y \in [p,p]$
(where we recall that $r = |f'(-1)|^{-1} \in (0,1)$), which falls under (H2)(ii).
The proof that (H3) holds is similar to the argument for AFN maps in Section~\ref{subsec:afn}.
\end{rmk}

\subsection{Fibonacci unimodal maps}

In \cite{BT15}, a family of countably piecewise linear unimodal maps with Fibonacci combinatorics
is presented, and their thermodynamic properties are studied.
The original motivation to study Fibonacci unimodal maps is their unusual Lebesgue ergodic
behaviour and (for sufficiently large critical order) the existence of a
Cantor set that is not Lyapunov stable (as is the case with infinitely renormalisable unimodal maps)
but attracts Lebesgue-a.e.\ point \cite{BKNS96}.
The countably piecewise linear model is used to deal with the otherwise severe
problems of distortion control. It is comparable to the Gaspard-Wang \cite{GW88} countably piecewise
linear map as model for the smooth Pomeau-Manneville map.

Let $\s_0, \s_1, \s_2, \s_3, \s_4, \ldots = 1,2,3,5,8, \dots$ be the Fibonacci numbers.
A Fibonacci unimodal map $f:[0,1] \to [0,1]$ with critical point $c = \frac12$, 
has two sequences of precritical points $z_k \nearrow c$ and $\hat z_k \searrow c$ such that
$f^{\s_k}$ maps both $[z_k, c]$ and $[c, \hat z_k]$ monotonically onto
$[c, f^{\s_k}c]$ whereas $f^{\s_k}(z_{k+1}) = f^{\s_k}(\hat z_{k+1}) = z_{k-1}$ or $\hat z_{k-1}$.
Therefore we can construct a Markov induced map 
\begin{equation}\label{eq:F}
F = f^{\r} \quad \text{ where } \r(x) = \s_k \text{ if } 
x \in \hat V_{k+1} := [z_k, z_{k+1}) \cup (\hat z_{k+1}, \hat z_k],
\end{equation}
see Figure~\ref{fig1}.
The countably piecewise linear Fibonacci map is constructed to be linear on each interval
$[z_k, z_{k+1})$ and $(\hat z_{k+1}, \hat z_k]$ and so that the induced map $F$ also
has linear branches, see \cite[Section 2]{BT15}.
In fact, for any $\lambda \in (0,1)$ one can construct a countably piecewise linear Fibonacci map $f_\lambda$ 
such that $|z_k-c| = |\hat z_k-c| = \frac12 \lambda^k$ and the induced map $F_\lambda = f_\lambda^{\r}$ 
has slopes $F'_\lambda = \lambda^{-1}(1-\lambda)^{-1}$ wherever defined on $[z_1, \hat z_1]$.
Furthermore, $F_\lambda$ is two-to-one semiconjugate to the map
\begin{equation}
T_\lambda(x):= \begin{cases}\ 
\frac{1-x}{1-\lambda} & \text{ if } x\in V_1 := (\lambda,1],\\[2mm]
\ \frac{\lambda^{n-1}-x}{\lambda(1-\lambda)} & \text{ if } x\in V_n := (\lambda^n, \lambda^{n-1}], \quad n \ge 2,
\end{cases}
\label{eq:tildeStrVogt}
\end{equation}
and isomorphic to the map studied by Stratmann \& Vogt \cite{SV97} and later in \cite{BT12} 
(see Figure~\ref{fig1}).  Note that $\log F_\lambda$ has zero variations.

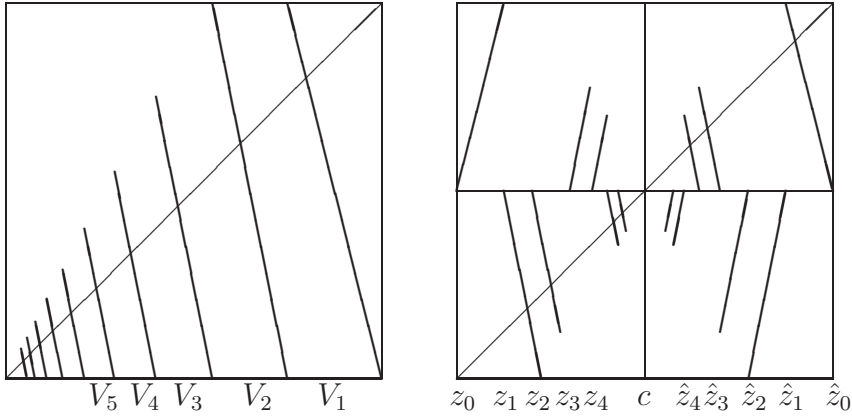
\begin{figure}[ht]
\unitlength=5mm
\begin{picture}(22,12)(-3.5,-0.5) \let\ts\textstyle
\thinlines
\put(0,0){\line(1,0){10}}\put(0,10){\line(1,0){10}}
\put(0,0){\line(0,1){10}} \put(10,0){\line(0,1){10}}
\put(0,0){\line(1,1){10}}
\thicklines
\put(10,0){\line(-1,4){2.5}} \put(8.3,-0.7){$\tiny V_1$}
\put(7.5,0){\line(-1,5){2}} \put(6.3,-0.7){$\tiny V_2$}
\put(5.5,0){\line(-1,5){1.5}}  \put(4.45,-0.7){$\tiny V_3$}
\put(4.0,0){\line(-1,5){1.1}}  \put(3.3,-0.7){$\tiny V_4$}
\put(2.9,0){\line(-1,5){0.8}}  \put(2.2,-0.7){$\tiny V_5$}
\put(2.1,0){\line(-1,5){0.58}}
\put(1.52,0){\line(-1,5){0.425}}
\put(1.095,0){\line(-1,5){0.304}}
\put(0.791,0){\line(-1,5){0.219}}
\put(0.572,0){\line(-1,5){0.1582}}
\put(0.4138,0){\line(-1,5){0.1164}}
\put(0.2974,0){\line(-1,5){0.08276}}
\thinlines
\put(12,0){\line(1,0){10}}\put(12,10){\line(1,0){10}}
\put(12,0){\line(0,1){10}} \put(22,0){\line(0,1){10}}
\put(17,0){\line(0,1){10}} \put(12,5){\line(1,0){10}}
\put(12,0){\line(1,1){10}}
\thicklines
\put(20.75,10){\line(1,-4){1.25}} \put(21.8,-0.7){$\tiny \hat z_0$}
\put(12,5){\line(1,4){1.25}} \put(11.8,-0.7){$\tiny z_0$}
\put(19.75,0){\line(1,5){1}} \put(20.55,-0.7){$\tiny \hat z_1$}
\put(13.25,5){\line(1,-5){1}} \put(12.95,-0.7){$\tiny z_1$}
\put(19,1.25){\line(1,5){0.75}} \put(19.55,-0.7){$\tiny \hat z_2$}
\put(14,5){\line(1,-5){0.75}} \put(13.8,-0.7){$\tiny z_2$}
\put(19,5){\line(-1,5){0.55}} \put(18.55,-0.7){$\tiny \hat z_3$}
\put(15,05){\line(1,5){0.55}} \put(14.6,-0.7){$\tiny z_3$}
\put(18.45,5){\line(-1,5){0.4}} \put(17.8,-0.7){$\tiny \hat z_4$}
\put(15.6,5){\line(1,5){0.4}} \put(15.35,-0.7){$\tiny z_4$}
\put(18.05,5){\line(-1,-5){0.29}}
\put(16,5){\line(1,-5){0.29}}
\put(17.76,5){\line(-1,-5){0.2125}}
\put(16.29,5){\line(1,-5){0.2125}}  \put(16.8,-0.7){$\tiny c$}
\end{picture}
\caption{The maps $T_\lambda:[0,1] \to [0,1]$ and $F_\lambda:[z_0, \hat z_0]
  \to [z_0, \hat z_0]$. }
\label{fig1}
\end{figure}
Lebesgue measure is dissipative if $\lambda > \frac12$ (see \cite[Theorem A]{BT15}).
The tail estimated in \eqref{eq:Fibotail} produces a conservative $f_\lambda$-invariant
measure that is finite if $\lambda \in (0, \frac{2}{3+\sqrt{5}})$ and infinite sigma-finite
if $\lambda \in (\frac{2}{3+\sqrt{5}}, \frac12)$.  

We summarise some of the main results of \cite{BT15} on
the shape of the pressure of potential $\phi_t = -t\log f_\lambda'$
and induced potential $\bar\phi_t = -t \log F_\lambda$:

\begin{thm}\cite[Theorem A]{BT12} and \cite[Theorem D]{BT15}\label{thm:bt2}
For $\lambda^t \leq \frac12$, 
$$
P(\bar\phi_t) = \log\left( \frac{(1-\lambda)^t}{1-\lambda^t} \right).
$$
For $\lambda \in (\frac{2}{3+\sqrt{5}}, \frac12)$ and golden mean $\GR = \frac{1+\sqrt{5}}{2}$, 
 there are constants $C_0, C_1$ such that
 \begin{equation}\label{eq:fibopressure}
 C_0 (1-t)^{\frac{\log \GR}{\log R}} 
 \leq P(\phi_t) 
 \leq  C_1 (1-t)^{\frac{\lambda \log \GR}{2t(1-2\lambda)}}
 \end{equation}
for $t \leq 1$, and 
$R = \frac{\left( 1 - \sqrt{1-4\lambda^t(1-\lambda)^t}\right)^2}{4\lambda^t (1-\lambda)^t}$.
Finally, $P(\phi_t) = 0$ for $t \geq 0$.
\end{thm}

The induced system has no phase transition at $t = 1$; instead $t \mapsto P(\bar\phi_t)$ 
is real analytic near $t = 1$ with a strictly negative slope.  We will first focus on 
the thermodynamics of $T_\lambda$ before moving on to the full Fibonacci case.

\subsubsection{Stratmann-Vogt example in isolation}

Here we consider the Stratmann-Vogt map $T=T_\lambda$ 
forgetting that it came from a unimodal map 
(so, for example, the Fibonacci numbers do not feature here).  As shown in \cite{BT12}, 
for $\lambda=1/2$ the equilibrium measure $\mu_t$ for $\phi = -\log|T'|$ is null recurrent. 
The idea here is to induce to a first return map to $V_1$.

As in  \cite{BT12}, from a topological point of view, the map $T$ has a natural Markov coding with 
transition matrix
$$
A = 
\left( \begin{array}{ccccccc}
1 & 1 & 1 & 1 & 1 & \hdots   & \hdots \\
1 & 1 & 1 & 1 & 1 & \hdots   & \hdots \\
0 & 1 & 1 & 1 & 1 & 1 & \hdots \\
0 & 0 & 1 & 1 & 1 & 1 & \hdots \\
\vdots & \vdots & 0  & 1 & 1 & 1 & \hdots \\
\vdots & \vdots & \vdots  & \vdots   & \vdots   & \vdots   & \ddots
\end{array} \right).
$$
Let $\rho$ be the first return time of $T$ to $V_1$ and denote the domains of the first return 
map $T^\rho$ by $\{Y_i\}_i$.  For $Y_i$ with $\rho|_{Y_i}=n$ we 
have $|Y_i|=\lambda \left[\lambda(1-\lambda)\right]^{n-1}=\frac2{4^n}$.  
Moreover, since the branches are linear, the equilibrium measure $\bar\mu$ for the return map $T^\rho$ is precisely 
normalised Lebesgue measure on $V_1$.  So to compute the measure of the tails, we only need to 
estimate the number $\#\{\tau=n\}$ of intervals $Y_i$ with $\rho|_{Y_i} = n$.  
But this can be computed explicitly.

The only way a point can return to $V_1$ in one step is if either it started in $V_1$ or it lies in $V_2$.  
Clearly there is only one domain  with inducing time $1$.  For $\rho=n\ge 2$, the problem boils down  to how many ways 
there are of entering $V_2$ without hitting $V_1$.

\begin{lemma}\label{lem:catalan}
The number $\#\{\rho=n\}$ is the $n$th Catalan number $C_n = \frac{1}{n+1} {2n \choose n}$.
\end{lemma}

\begin{proof}
A Dyck word of length $2n$, $(x_1, \ldots, x_{2n})$ has $n$ ups and $n$ downs and at any 
truncation $(x_1, \ldots, x_k)$ for $k\le n$, never has more downs than ups.  There are $C_n$ Dyck words of 
length $2n$.  We will show that this is in bijection with our paths $(V_2, V_{i_1}, \ldots, V_{i_{n-1}},V_2)$. 
We can code each path from $V_2$ back to $V_2$ by a Dyck word.  
For a  Dyck word $(x_{j_1}, \ldots, x_{j_n})$, we call a maximal subword $(u, u,\ldots, u, d, d, \ldots, d)$ 
a  \emph{peak} (namely, the letter preceding this is either empty or $d$, and the letter succeeding it is 
either empty or $u$).  For the first peak, if it consists of $k_1$ consecutive ups and then $\ell_1$ 
consecutive downs, write $m_1=2+k_1-\ell_1$. Then we get the code $(V_2, V_{2+k_1}, V_{2+k_1-1}, V_{2+k_1-2}, \ldots, V_{m_1})$.  
For the next peak, consisting of $k_2$ consecutive ups and then $\ell_2$ consecutive downs, write $m_2=m_1+k_2-\ell_2$ and we append the code $(V_{m_1+k_2}, V_{m_1+k_2-1}, V_{m_1+k_2-2}, \ldots, V_{m_2})$.  Continuing in this way, and appending $V_2$ at the very end, 
we obtain a code $(V_2, V_{i_1}, \ldots, V_{i_n},V_2)$.  
In other words, the Dyck word is obtained as the concatenation
\begin{equation}\label{eq:dyck}
 u^{1+i_1-2} d u^{1+i_2-i_1}  d u^{1+i_3-i_2}  d  \dots u^{1+2-i_{n-1}} d.
\end{equation}
It is easy to see that this coding is a bijection.
\end{proof}

We know that $C_n\sim \frac{4^n n^{-3/2}}{\sqrt\pi}$.
In fact, the precise asymptotics (see \cite[Chapter V.2]{FlaSed09}) is:
$$
C_n =\frac{4^n}{\sqrt\pi }\left( n^{-\frac32} - 
\frac98 n^{-\frac52} +\frac{145}{128} n^{-\frac72} -\cdots\right).
$$
Since for any domain of the first return map $Y_i$ with $\rho|_{Y_i}=n$ we have 
$|Y_i|=\frac{|V_1|}{1-\lambda} \left[\lambda(1-\lambda)\right]^{n-1}$ and since 
the equilibrium measure $\bar\mu$ for the return map is precisely normalised $\bar\phi$-conformal 
measure on $V_1$ for $\bar\phi = \sum_{j=0}^{\rho-1} \phi \circ T^j$, this means that
\begin{eqnarray}
\bar\mu(\rho > n) &=& \frac{|V_1|}{|V_1| (1-\lambda) }
\sum_{k > n} C_k \left[\lambda(1-\lambda)\right]^{k-1} \nonumber\\
&\sim&
\frac{1}{\sqrt\pi(1-\lambda)}\sum_{k > n} \left[4 \lambda(1-\lambda)\right]^k k^{-\frac32}. \label{eq:Cat_tail}
\end{eqnarray}
Hence, for $\lambda < \frac12$, it is has exponential tails.

In the interesting null recurrent case, 
$\lambda = \frac12$, so $|Y_i| = \frac{2|V_1|}{4^n}$, 
the tail is 
$$
\bar\mu(\rho>n) = c_1n^{-\frac12}+c_2n^{-\frac32}+O(n^{-\frac52}),
$$ 
which fits in with (H1)(b).
The induced map consists of countably many linear onto branches, so it satisfies all 
the other conditions of Section~\ref{sec:abstsu} and Theorem~\ref{thm:PbarP} applies.

\begin{prop}\label{prop:psilogtau3}
Consider $T_{1/2}$.  
Suppose that $\psi:[0,1]\to \R$ is piecewise $\kappa$-H\"older for some $\kappa>0$, 
such that $\psi(x) = C' > 0$ for $x \in V_0$ and $\psi(x) \leq -b < 0$ for $x \in [0,1] \setminus V_0$ 
and $P(\phi+s\psi)>0$. Then, provided $C'$ is sufficiently large compared to $b$,
$$
P(\phi+s\psi) = (CP(\overline{\phi+s\psi}))^2(1+Q(s))
$$
as $s \to 0$ for $C$ and $Q$ as in Theorem~\ref{thm:PbarP}. 
\end{prop}

The proof goes as in Proposition~\ref{prop:psilogtau}, except that the assumption $\bar\psi\in SV$ does not 
require a particular H\"older exponent on $\psi$ as $T_{1/2}$ is uniformly expanding.

\begin{rmk}
By \cite[Theorem 2]{BT12}, $P(\phi_t) = (1-t) \log 4$ for $\phi_t = -t \log |T'|$ and $t\le1$.
\end{rmk}

\subsubsection{The Fibonacci case}

We are interested in the infinite acim $\mu=\mu_1$ which, as in Theorem~\ref{thm:bt2}, 
we see whenever $\lambda \in (\frac{2}{3+\sqrt{5}}, \frac12)$.  Since the pressure of the relevant 
potential $\phi=\log|f'|$ is zero, the form of the induced measure for $T_\lambda$ can be taken 
directly from \cite{BT15}.
The inducing time $\r$ is a lacunary sequence ($\r$ takes only Fibonacci numbers as values).  
From \cite[Proposition 1]{BT12} we have the small tail estimates for the 
equilibrium measure $\mu_{\bar\phi}$ for $\bar\phi = \sum_{j=0}^{\r-1} \phi \circ f^j$:
\begin{equation} \label{eq:Fibotail}
\mu_{\bar\phi}(\r = \s_k) = \frac{1-2\lambda}{1-\lambda} \left( \frac{\lambda}{1-\lambda} \right)^k.
\end{equation}
The big tail estimates
\begin{equation} \label{eq:Fibotail2}
\mu_{\bar\phi}(\r \geq \s_k) =  \left( \frac{\lambda}{1-\lambda} \right)^k
\sim \GR_0 \s_k^{(\log \frac{\lambda}{1-\lambda} )/ \log \GR} = \GR_0 \s_k^{-\beta}
\end{equation}
are not regularly varying, but at least polynomial with exponent 
$\beta := (\log \frac{1-\lambda}{\lambda} )/ \log \GR$.
Here $\GR_0 =\frac{\lambda}{1-\lambda} \frac{3+\sqrt{5}}{2+\sqrt{5}}^{\beta}$ based on Binet's
formula $\s_k = \frac{1}{\sqrt{5}} (G^{k+2} - (-1/G)^{k+2})$,
and note that $\beta \in (0,1)$ for $\lambda \in (\frac{2}{3+\sqrt{5}}, \frac12)$.

\begin{prop}
For the countably piecewise linear Fibonacci map $f_\lambda$ 
for $\lambda \in (\frac{2}{3+\sqrt{5}}, \frac12)$ and potential $\phi_t = -t\log|f'_\lambda|$ 
we have
$$
P(\phi_t) \asymp (1-t)^{1/\beta} \quad \text{ as } \quad t \nearrow 1.
$$
\label{prop:fibpress}
\end{prop}

The main challenge here is proving the following theorem, which is of independent interest: the proof of the above proposition then follows similarly to the previous cases.  In order to state the theorem let $\mu_{\bar\phi}$ denote the measure for $F=f^\r$ and
$\mu_{\hat\phi}$ denote the measure for $\hat F = F^\rho$,
Let $\bar\rho:=\int_Y \rho~d\mu_{\bar\phi}<\infty$.  
Recall that $\beta = \log[(1-\lambda)/\lambda] /\log \GR$.

\begin{thm}\label{thm:Fibo}
Let $\mu_{\hat\phi}$ be the reinduced measure, associated to $F^\rho = (f^\r)^\rho = f^\tau$. 
There are constants $0 < C_2 < C_1$ such that
$$
C_2 n^{-\beta} \leq \mu_{\hat\phi}(\tau>n) \leq C_1 n^{-\beta}.
$$
\end{thm}

\begin{proof}
From \cite[Lemma A.1]{BT16} we have
\begin{align*}
\mu_{\hat\phi}(\tau>n) &-\bar\rho\, \mu_{\bar\phi}(\r>n) \\
& = \sum_{k\ge 0}\left(\int_{\{\tau=\r_{k+1}\}} 1_{\{n \ge \r>n - \r_k\}}\circ f^{\r_k}~d\mu_{\hat\phi} 
- \int_{\{\tau>\r_{k+1}\}} 1_{\{ \r>n\}}\circ f^{\r_k}~d\mu_{\hat\phi}\right).
\end{align*}
Contrary to the statement of  \cite[Lemma A.1]{BT16}, the domain $[z_0, \hat z_0]$ of our map $f^\sigma$
is a proper superset of the domain $V_1$ of $(f^\sigma)^\rho$. However, the proof of Lemma A.1.\ relies 
only on the representation $\pi^{-1}([z_0, \hat z_0])$ of $[z_0, \hat z_0]$ inside the tower over
$V_1$, and taking that point of view, there is no difference in the proof.
\\[2mm]
{\bf Upper bound:} Take $N = \lfloor \frac{-\beta }{\log 4\lambda(1-\lambda) } \log n \rfloor$.
It suffices to estimate the sum 
$$
J_1 = \sum_{k=0}^N \int_{\{\tau=\r_{k+1}\}} 1_{\{n \ge \r>n- \r_k\}}\circ f^{\r_k}~d\mu_{\hat\phi}.
$$
because the remaining sum 
$\sum_{k=N+1}^\infty \int_{\{\tau=\r_{k+1}\}} 1_{\{n \ge \r>n- \r_k\}}\circ f^{\r_k}~d\mu_{\hat\phi} = O(N^{-\beta}).$
Since $\tau(y) = \r_{k+1}(y)$ implies that $f^{\r_k}(y) \in \hat V_2 \cap F^{-1}(\hat V_1)$, 
so $\r(f^{\r_k}(y)) = 2$ and therefore we have additionally $\r_k > n-2$.
Hence the summand in the above formula consists of all $k$-paths $\hat V_2$ to $\hat V_2$
avoiding $\hat V_1$, and for which $\tau_k > n-2$.
As in the proof of Lemma~\ref{lem:catalan}, such paths are in one-to-one correspondence with 
Dyck words of length $2k$, or equivalently, with random walks on $\N_0 := \N \cup \{ 0 \}$ of length $2k$, 
starting and ending at $0$. 
If the maximum of such a path is $M < (\log(n-1)-\log k)/\log G$, then
$\r_k \leq k \s_M \leq n-2$ so the condition $\r_k > n-2$ implies that 
$M \geq (\log(n-1)-\log k)/\log G$.
This also implies that $k \geq N_0 := \lfloor (\log n)/\log G \rfloor$ since 
otherwise the path is too short to return from $M$ to $0$.
The theory of random walks on $\N_0$ 
(specifically, the reflection principle) says that the number of $2k$-paths
from $0$ back to $0$ with maximum $\geq M$ is equal to the number of $2k$-paths
from $0$ to $2M$.
Indeed, such a path must have a last instance at which it takes the value $M$, and then
we reflect the remaining path, namely from $M$ to $0$, to a path from
$M$ to $2M$.
The number of such paths is ${2k \choose M+k}$.
Since each such path corresponds to one domain of $\hat F$
of length $[\lambda(1-\lambda)]^k$, we obtain that the measure of the set of
points with this property (for fixed $M$) is bounded by
\begin{equation}\label{eq:largestk}
\Delta_k := {2k \choose M+k}[\lambda(1-\lambda)]^k  \sim
\frac{(2k)^{2k} [\lambda(1-\lambda)]^k }{ (k+M)^{k+M} (k-M)^{k-M} }\,
\sqrt{\frac{k}{\pi(k^2-M^2)}},
\end{equation}
where we used Stirling's formula.
Taking the logarithm of this expression, and setting its derivative 
$\frac{d}{dk}\log(\Delta_k) = 0$,
we find
$$
0 = \log\left(\frac{2k}{k+M} \frac{2k}{k-M}\right) + \log \lambda(1-\lambda)
- \frac{1}{2\pi} \frac{k^2+M^2}{k(k^2-M^2)}.
$$
For  $k = k_0 := \frac{M}{1-2\lambda}$,
so $\frac{k+M}{2k} = 1-\lambda$ and $\frac{k-M}{2k} = \lambda$,
the two logarithms in this expression cancel, and the third term is small
for large $M$, so we expect that $\Delta_k$ is maximal for $k$ close to $k_0$.
Comparing the binomial coefficients directly, one can check that
$\max\{ \Delta_j : j \neq k_0, k_0-1 \} < \Delta_{k_0} < \Delta_{k_0-1} =
 \frac{2k_0}{2k_0-1} \Delta_{k_0}$, so the global maximum is $\Delta_{k_0-1}$.
For the value $k_0 = M/(1-2\lambda)$, we have $M = N_0 = (\log n)/\log \GR$, and
\begin{eqnarray}\label{eq:Deltak0}
\Delta_{k_0-1} &=& \frac{2k_0}{2k_0-1} \Delta_{k_0} = \frac{2k_0}{2k_0-1} \sqrt{\frac{1-2\lambda}{4\lambda(1-\lambda)\pi M} }
[\lambda/(1-\lambda) ]^M \nonumber \\
&=& \frac{2k_0}{2k_0-1}\sqrt{\frac{(1-2\lambda) \log \GR}{4\lambda(1-\lambda)\pi} }\, 
\frac{n^{-\beta}}{\sqrt{\log n}}.
\end{eqnarray}
Therefore
$$
\frac{1}{\sqrt{\log n}} n^{-\beta} \ll J_1 \asymp \sum_{k=N_0}^N \Delta_k \ll \sqrt{\log n}\, n^{-\beta},
$$
but we will improve this to $J_1 \asymp n^{-\beta}$.

Let $M_0$ be the largest even integer $\leq \sqrt{M}$.
We claim that there exist $0 < \theta < \theta'$ and $K < K'$ such that
\begin{eqnarray*}
K e^{\theta/2} n^{-\beta} \leq \sum_{j = k_0}^{k_0+ M_0 - 1} \Delta_j && \sum_{j = k_0+\ell M_0}^{k_0+(\ell+1) M_0 - 1} \Delta_j 
\leq K e^{-\theta \ell/2} n^{-\beta},\\
K e^{-\theta'/2} n^{-\beta} \leq \sum_{j = k_0-M_0}^{k_0 - 1} \Delta_j
&\qquad& \sum_{j = k_0-(\ell+1) M_0}^{k_0-\ell M_0 - 1} \Delta_j \leq 
K' e^{-\theta \ell/2} n^{-\beta},
\end{eqnarray*}
for all $\ell \geq 0$ such that the index $j \in [M, N]$.

To prove this, first compute (recalling that $1-\lambda = \frac{k_0+M}{2k_0}$
and $\lambda = \frac{k_0-M}{2k_0}$)
\begin{eqnarray*}
\Delta_{k_0+M_0} &=& \Delta_{k_0} [\lambda(1-\lambda)]^{M_0} \cdot 
\frac{(2k_0+2)(2k_0+4) \cdots (2k_0+2M_0)}{(k_0+M+1)(k_0+M+2) \cdots (k_0+M+M_0)} \\
&& \qquad \cdot \frac{(2k_0+2-1) (2k_0+4-1) \cdots (2k_0+2M_0-1)}
 {(k_0-M+1)(k_0-M+2) \cdots (k_0-M+(M_0-1)) } \\
&=& \Delta_{k_0} \prod_{j=1}^{M_0} \left( (1-\lambda) \frac{2(k_0+j)}{k_0+M+j} \right)
\, \prod_{j=1}^{M_0} \left( \lambda \frac{2(k_0+j)-1}{k_0-M+j} \right)\\
 &=& \Delta_{k_0} \prod_{j=1}^{M_0} \frac{1+\frac{2j}{2k_0}}{1+\frac{j}{k_0+M}}
 \, \prod_{j=1}^{M_0} \frac{1+\frac{2j-1}{2k_0}}{1+\frac{j}{k_0-M}}\\
&\sim& \Delta_{k_0} \prod_{j=1}^{M_0} \left(1 + j \frac{M}{k_0(k_0+M)} \right)
 \prod_{j=1}^{M_0} \left(1-j \frac{M}{k_0(k_0-M)} \right)\\
 &\sim& \Delta_{k_0} \prod_{j=1}^{M_0} \left(1 - j \frac{2M^2}{k_0(k_0^2-M^2)} \right).
\end{eqnarray*}
Recall that $k_0 = \frac{M}{1-2\lambda}$, so 
$\frac{2M^2}{k_0(k_0^2-M^2)} = \frac{(1-2\lambda)^3}{2M\lambda(1-\lambda)} =: \frac{\theta}{M}$.
Therefore
\begin{eqnarray*}
\Delta_{k_0+M_0} &\sim& \Delta_{k_0} \exp 
\sum_{j=1}^{M_0} \log \left(1-\frac{j}{M} \frac{(1-2\lambda)^3}{2\lambda(1-\lambda)}\right) \\
&\sim&\Delta_{k_0} \exp \left( -  \frac{\theta}{M}  \sum_{j=1}^{M_0} j \right)
= \Delta_{k_0} \exp \left( -  \frac{\theta}{M} \frac{(M_0+1)M_0}{2} \right)
\sim \Delta_{k_0} e^{-\theta/2}.
\end{eqnarray*}
Since $\Delta_{k_0+j}$ is decreasing in $j$, we have
$$
K e^{-\theta/2} n^{-\beta}
= M_0 \Delta_{k_0} e^{-\theta/2} \leq \sum_{j=0}^{M_0-1} \Delta_{k_0+j} \leq M_0 \Delta_{k_0} = K n^{-\beta},
$$
for $K = \sqrt{\frac{1-2\lambda}{4\lambda(1-\lambda)\pi} }$ as in \eqref{eq:Deltak0}.
A similar argument gives a $\theta'$ such that
$$
K e^{-\theta'/2} n^{-\beta}
= M_0 \Delta_{k_0} e^{-\theta'/2} \leq \sum_{j=1}^{M_0} \Delta_{k_0-j} \leq M_0 \Delta_{k_0} = 
K n^{-\beta},
$$
and also
$$
\sum_{j=\ell M_0}^{(\ell+1)M_0-1} \Delta_{k_0+j} \leq K e^{\ell \theta/2}  n^{-\beta}
\quad \text{ and } \quad
\sum_{j=(\ell+1) M_0}^{\ell M_0-1} \Delta_{k_0-j} \leq K e^{\ell \theta'/2}  n^{-\beta}, 
$$
as long as $j \in [M, N]$. This proves the claim.

Finally, using geometric series, we can find $0 < K_2 < K_1$ such that
$$
K_2 n^{-\beta} \leq 
J_1 = \sum_{j=M}^N \Delta_j \sim \sum_{\ell = (k_0-M)/M_0}^{(N-k_0)/M_0}
\sum_{j=k_0+\ell M_0}^{k_0+(\ell+1)M_0-1} \Delta_j \leq K_1 n^{-\beta}.
$$
\\[2mm]
{\bf Lower bound:} We estimate
$$
J_2 := \sum_{k=0}^N\int_{\{\tau>\r_{k+1}\}} 1_{\{ \r>n \}} \circ f^{\r_k}~d\mu_{\hat\phi}.
$$
If $y$ is such that $\s_M = \r \circ f^{\r_k}(y)$ for $n \geq 3$, then automatically 
$\rho > k+1$, and the domain of $f^{\r_{k+1}} = F^{k+1}$ containing $y$ maps
to an interval of length $O(\lambda^{M-1})$.
This is because $F^{k+1}$ maps the corresponding domain
onto one of the two components of $\cup_{m \geq M-1} \hat V_m$.
The slope of $F^k$ is $\lambda [\lambda(1-\lambda)]^k$, so the domain itself has length
$O(\lambda^M [\lambda(1-\lambda)]^k)$.

Using Dyck words, i.e., the codings from \eqref{eq:dyck} (or equivalently paths of the standard random walk on $\N_0$),
the path from $V_2$ to $V_m$, $m \geq M-1$, corresponds to a Dyck word
$x_1 \dots x_{2k+m}$ with $k+m$ ``ups'' and $k$ ``downs''.
Since the number of ``ups'' always exceeds the number of ``downs`` in all
subwords $x_1 \dots x_j$, the reflection principle gives that there
are ${2k+m \choose k+m} - {2k+m \choose k+m+2} =
{2k+m \choose k+m} (1-\frac{k(k+1)}{(k+m+2)(k+m+1)}) < {2k+m \choose k+m}$
such words.
Since $\r > n$ implies $m \geq N_0 := \lfloor (\log n)/\log \GR\rfloor$, we get
$$
J_2 \ll \sum_{k = N_0}^N \sum_{m \geq N_0} {2k+m \choose k+m} \lambda^m [\lambda(1-\lambda)]^k.
$$
Now
\begin{eqnarray*}
 {2k+m \choose k+m} &=&  
 {2k+m \choose \frac{2k+m}{2}} \frac{ {2k+m \choose k+m} }{  {2k+m \choose \frac{2k+m}{2}} } \\
 &\leq& {2k+m \choose \frac{2k+m}{2}} 
 \frac{(k+1)(k+2) \cdots (k+\frac{m}{2})}{(k+\frac{m}{2}+1)(k+\frac{m}{2}+2) \cdots (k+m)}\\
 &\leq& \sqrt{\frac{2}{\pi}} \frac{2^{2k+m}}{\sqrt{2k+m}} 
 \frac{(k+1)(k+2) \cdots (k+m/2)}{(k+\frac{m}{2}+1)(k+\frac{m}{2}+2) \cdots (k+m)}. 
\end{eqnarray*}
Therefore
$$
J_2 \ll \sum_{k = N_0}^N [4\lambda(1-\lambda)]^k \sum_{m \geq N_0} \frac{(2\lambda)^m}{\sqrt{2k+m}}
\frac{(k+1)(k+2) \cdots (k+\frac{m}{2})}{(k+\frac{m}{2}+1)(k+\frac{m}{2}+2) \cdots (k+m)}.
$$
The largest value of the summand is achieved at $k = N_0$, and is equal to
\begin{align*}
[4\lambda(1-\lambda)]^{N_0} & \sum_{m \geq N_0} \frac{(2\lambda)^m}{\sqrt{2N_0+m}}
\frac{(N_0+1)(N_0+2) \cdots (N_0+\frac{m}{2})}{(N_0+\frac{m}{2}+1)(N_0+\frac{m}{2}+2) \cdots (N_0+m)} \\
& = [4\lambda(1-\lambda)]^{N_0} \sum_{m \geq N_0}  \frac{(2\lambda)^m}{\sqrt{2N_0+m}}
\prod_{j=1}^{N_0/2} \frac{N_0+j}{N_0+\frac{m}{2}+j}.
\end{align*}
We estimate the product in this expression with $m = N_0$, using a Riemann sum, as
\begin{eqnarray*}
\prod_{j=1}^{N_0/2} \frac{N_0+j}{N_0+\frac{N_0}{2}+j} &\leq& 
\sqrt{2/3}^{N_0} \exp \sum_{j=1}^{N_0/2} \log(1+\frac{j}{3N_0}) \\ 
&\sim& \sqrt{2/3}^{N_0} \exp\left(\frac{N_0}{2} \int_0^1 \log(1+\frac{x}{6}) \, dx\right) \leq
\sqrt{2/3}^{N_0}  e^{0.04 N_0}.
\end{eqnarray*}
Therefore
$$
J_2 \ll \frac{N-N_0}{\sqrt{2N_0}} \left( 4\lambda^2(1-\lambda) \sqrt{\frac23} e^{1/12}\right)^{N_0} 
\leq \frac{1}{\sqrt{\log n}} n^{\frac{1}{\log \GR} \log[8\lambda^2(1-\lambda) \sqrt{\frac23} e^{0.04}]} < n^{-\beta},
$$
provided $8\lambda^2(1-\lambda) \sqrt{\frac23} e^{0.04} < \log \frac{\lambda}{1-\lambda}$,
or equivalently $8\lambda(1-\lambda)^2 \sqrt{\frac23} + 0.04 < 1$.
The maximum of $\lambda^2(1-\lambda)$ on $[\frac{2}{3+\sqrt{5}}, \frac12]$ 
occurs at $\lambda = \lambda_0 := \frac{2}{3+\sqrt{5}}$, and then
$8\lambda_0(1-\lambda_0)^2 \sqrt{\frac23} + 0.04 \approx 0.93$.
Hence, there exists $\beta' > \beta$ such that $J_2 \ll n^{-\beta'}$ as $n \to \infty$.

Combining the estimates for $J_1$ and $J_2$ with \eqref{eq:Fibotail2}, we find 
$\mu_{\hat\phi}(\tau > n) = \bar\rho \, \mu_{\bar\phi}(\r > n) + J_1 - J_2 =  O(n^{-\beta})$,
and this proves the theorem.
\end{proof}

\begin{pfof}{Proposition~\ref{prop:fibpress}}
The facts that the potential $-t\log|f'|$ is recurrent for $t\le 1$ and $P(-t\log|f'|)>0$ 
for $t<1$ follow from \cite{BT07, BT15}.  
 The induced map $F = f^{\r}$ has good liftability properties because
 every $f$-invariant measure is liftable, see \cite{Kellift}, except the Dirac measure $\delta_0$, 
 but this measure is supported outside the core $[f^2(c), f(c)]$.  
 So all of our potentials are liftable for $F$.  Indeed, 
 the induced potentials $-t\log|F'|-\r P(-t\log|f'|)$ for $F$ are also recurrent since as in \cite{BT15} 
 these potentials all 
 have equilibrium states for $t\le 1$. Moreover, the variations of the induced potential are zero.
 
However, $F$ doesn't have the BIP property, and therefore 
we induce again, as a first return to $\hat V_1:=[z_0, z_1)\cup(\hat z_1, \hat z_0]$. 
For $\rho(y) = \min\{ n \geq 1 : F^n(y) \in \hat V_1\}$ the map $\hat F = F^\rho: \hat V_1 \to \hat V_1$
is then a Markov map with linear branches that are onto a component of $\hat V_1$ (so the BIP property holds) 
and reinduced potential
$\hat \phi = \sum_{j=0}^{\rho-1} \bar\phi \circ F^j = -\log |\hat F'|$.
If we take a first return map for a Markov system with a recurrent potential with 
summable variations and zero pressure, then the potential induces to one which also has zero pressure 
(see \cite[Theorem 2]{Sarig01a}), so all of our potentials are liftable for $\hat F$ as well.

It remains check (H1)(a), but this follows from Theorem~\ref{thm:Fibo}.
\end{pfof}

\begin{rmk}
In order to compare this with the results in \cite{BT15} 
(as recalled in Theorem~~\ref{thm:bt2} we note that
analogous result holds for $\phi_t$-conformal measure (instead of Lebesgue), 
when we replace $\lambda$ with $\lambda^t$.

For $\lambda^t \leq \frac12$ and hence $\sqrt{1-4\lambda^t(1-\lambda)^t} \in [0,1)$
we have $2\sqrt{(1-\lambda)^t (1-\lambda^t)} \geq 1 \geq 1-\sqrt{1-4\lambda^t(1-\lambda)^t}$ for $t \leq 1$.
Squaring gives $\log \frac{1-\lambda^t}{\lambda^t} \geq \log R$
for $R$ as in Theorem~\ref{thm:bt2}.
Another simple calculation gives 
$\frac{2}{\lambda} (1-2\lambda) \geq \log \frac{1-\lambda^t}{\lambda^t}$.
Therefore, the exponent $1/\beta(t)$ lies indeed between the values given 
in \eqref{eq:fibopressure}.
In fact, the lower bound from \eqref{eq:fibopressure} is quite accurate, because
both $\lim_{t \nearrow 1} \log R \sim 2(1-2\lambda)$ and $\lim_{t \nearrow 1} 
\log \frac{1-\lambda^t}{\lambda^t} \sim 2(1-2\lambda)$ as $\lambda \nearrow \frac12$.
\end{rmk}


\begin{thebibliography}{BKNS96}

\bibitem[A97]{Aaronson} J.~Aaronson, 
\emph{An Introduction to Infinite Ergodic Theory,} 
Math.\ Surveys and Monographs \textbf{50}, Amer. Math. Soc., 1997.

\bibitem[AD01]{AaronsonDenker01} J.~Aaronson, M.~Denker, 
{Local limit theorems for partial sums of stationary
sequences generated by Gibbs-Markov maps}, 
\emph{Stoch. Dyn.\ } \textbf{1} (2001) 193--237.
  
\bibitem[BM89]{BM89} M.\ Benedicks, M.\ Misiurewicz,
  {Absolutely continuous invariant measures for maps with flat tops}, 
  \emph{Publ.\ Math.\ Inst.\ Hautes Etud.\ Sci.\ } \textbf{69} (1989) 203--213.
  
\bibitem[BGT]{BGT} N.~H.\ Bingham, C.~M.\ Goldie, J.~L.\ Teugels. 
\emph{Regular variation,}
Encyclopedia of Mathematics and its Applications \textbf{27}, 
Cambridge University Press, Cambridge, 1987. 

\bibitem[B75]{Bow75} R.\ Bowen,
{Equilibrium states and the ergodic theory of Anosov  diffeomorphisms,}
\emph{Lecture Notes in Mathematics}, \textbf{470} Springer-Verlag, Berlin-New York, 1975.  

\bibitem[BKNS96]{BKNS96} H.\ Bruin, G.\ Keller, T.\ Nowicki, S.\ van Strien,
{Wild Cantor attractors exist,}
\emph{Ann.\ of Math.\ }\textbf{143} (1996) 97--130.

\bibitem[BNT09]{BNT} H.\ Bruin, M. \ Nicol, D.\ Terhesiu,
{On Young towers associated with infinite measure preserving transformations,}
\emph{Stoch.\  and  Dyn.\ } \textbf{9} (2009) 635 -- 655.

\bibitem[BTe18]{BT16} H.\ Bruin, D.\ Terhesiu,
{Upper and lower bounds for the correlation function via inducing with general return times}, 
\emph{Ergod.\ Th.\ and Dyn.\ Sys.\ } \textbf{ 38} (2018) 34--62.

\bibitem[BTo09]{BT07} H.\ Bruin, M.\ Todd,
{Equilibrium states for interval maps: the potential $-t \log |Df|$}, 
\emph{Ann.\ Sci.\ Ecol.\ Norm.\ Sup.\ } \textbf{42} (2009) 559--600.

\bibitem[BTo12]{BT12} H.\ Bruin, M.\ Todd,
{Transience and thermodynamic formalism for infinitely branched interval maps},
\emph{Journ.\ London Math.\ Soc.\ } \textbf{86} (2012) 171--194.

\bibitem[BTo15]{BT15} H.\ Bruin, M.\ Todd,
{Wild attractors and thermodynamic formalism},
\emph{Monatshefte f\"ur Mathematik,} \textbf{178} (2015) 39--83. 

\bibitem[DS15]{DS15} N.\ Dobbs, M.\ Stenlund,
{Quasistatic dynamical systems,}
Ergod.\ Th.\ and Dyn.\ Sys.\ \textbf{37} (2017) 2556--2596. . 
 
\bibitem[F66]{Feller66} W.~Feller,
\emph{An Introduction to Probability Theory and its Applications, II}, 
Wiley, New York, 1966.

\bibitem[FF70]{FF70} M.\ Fisher, B.\ Felderhof,
{Phase transitions in one-dimensional cluster-interaction fluids IA,} 
\emph{Ann.\ Phys.\ } \textbf{58} (1970) 176--216.   

\bibitem[FS09]{FlaSed09} P.\ Flajolet, R.\ Sedgewick, 
\emph{Analytic combinatorics,} Cambridge University Press, Cambridge, 2009.
  
\bibitem[GW88]{GW88} P.\ Gaspard, X.\ J.\ Wang,
{Sporadicity: between periodic and chaotic dynamical behaviours,}
\emph{Proc.\ Nat.\ Acad.\ Sci.\ USA} \textbf{85} (1988) 4591--4595.
  
\bibitem[G04]{Gouezel04} S.~Gou{\"e}zel, 
{Sharp polynomial estimates for the decay of correlations},
\emph{Israel J.\ Math.\ } \textbf{139} (2004) 29--65.

\bibitem[G10]{Gouezel10} S.~Gou{\"e}zel, 
Characterization of weak convergence of Birkhoff sums for Gibbs-Markov maps,
\emph{Israel J.\ Math.\ } \textbf{180} (2010) 1--41. 

\bibitem[G11]{Gouezel11} S.~Gou{\"e}zel, 
{Correlations from large deviations in dynamical systems with infinite measure,}
\emph{Colloquium Math.\ } \textbf{125} (2011) 193--212

\bibitem[IT10]{IomTod10} G.\ Iommi, M.\ Todd,
{Natural equilibrium states for multimodal maps,}
\emph{Comm.\ Math.\ Phys.\ } \textbf{300} (2010) 65--94.

\bibitem[IT13]{IT13} G.\ Iommi, M.\ Todd,
{Thermodynamic formalism for interval maps: inducing schemes,}
\emph{Dyn.\ Sys.\ }  \textbf{28} (2013) 354--380.

\bibitem[Ka76]{Kato} T.~Kato.  \emph{Perturbation Theory for Linear Operators.}
Grundlehren \textbf{132} Springer, New York, 1976.
              
\bibitem[Kel89]{Kellift} G.\ Keller,
{Lifting measures to Markov extensions, } 
\emph{Monatshefte f\"ur Mathematik,} \textbf{108} (1989) 183--200.

\bibitem[LSV99]{LiveraniSaussolVaienti99} C.~Liverani, B.~Saussol, S.~Vaienti, 
{A probabilistic approach to intermittency}, 
\emph{Ergod.\ Th.\ and Dyn.\ Sys.\ } \textbf{19} (1999) 671--685.

\bibitem[L93]{L93} A.\ Lopes, 
{The zeta function, non differentiability of pressure, 
and the critical exponent of transition,} 
\emph{Adv.\ Math.\ } \textbf{101} (1993) 133--165.

\bibitem[MT12]{MT12} I.\ Melbourne, D.\ Terhesiu, 
{Operator renewal theory and mixing rates for dynamical
systems with infinite measure,}  
\emph{Invent.\ Math.\ } \textbf{1} (2012) 61--110.

\bibitem[MT13]{MT13} I.\ Melbourne, D.\ Terhesiu,
{First and higher order uniform ergodic theorems for dynamical
systems with infinite measure,} 
\emph{Israel J.\ Math.\ } \textbf{194} (2013) 793--830.

\bibitem[MTo04]{MelTor04} I.\ Melbourne, A.\ T\"or\"ok,
 {Statistical limit theorems for suspension flows,}
 \emph{Israel J.\ Math.\ } \textbf{144} (2004) 191--209.
 
\bibitem[PP90]{PP} W.\ Parry, M.\ Pollicott,
{Zeta functions and the periodic orbit structure of hyperbolic dynamics,} 
\emph{Ast\'erisque} \textbf{187--188} 1990.

\bibitem[S99]{Sarig99}  O.\ Sarig,
Thermodynamic formalism for countable Markov shifts, 
\emph{Ergod.\ Th.\ and Dyn.\ Sys.\ } \textbf{19} (1999)  1565--1593.

\bibitem[S01a]{Sarig01}  O.\ Sarig,
Thermodynamic formalism for null recurrent potentials, 
\emph{Israel J.\ Math.\ } \textbf{121} (2001) 285--311. 

\bibitem[S01b]{Sarig01a}  O.\ Sarig,
Phase transitions for countable topological Markov shifts,
 \emph{Comm.\ Math.\ Phys.\ } \textbf{217}   (2001) 555--577.

\bibitem[S02]{Sarig02} O.\ Sarig, 
{Subexponential decay of correlations,} 
\emph{Invent.\ Math.\ } \textbf{150} (2002) 629--653.

\bibitem[S06]{Sarig06} O.\ Sarig,
Continuous phase transitions for dynamical systems, 
\emph{Comm.\ Math.\ Phys.\ } \textbf{267} (2006)  631--667. 

\bibitem[SV97]{SV97} B.\ Stratmann, R.\ Vogt,
{Fractal dimension of dissipative sets,}
\emph{Nonlinearity} \textbf{10} (1997) 565--577.

\bibitem[Ta17]{T17} H.\ Takahasi,
{Equilibrium states at freezing phase transition in unimodal maps with flat critical point},
Preprint 2017 arXiv:1707.06435

\bibitem[Te15]{T15} D.~Terhesiu, 
{Improved mixing rates for infinite measure preserving transformations,}
\emph{Ergod.\ Th.\ and Dyn.\ Sys.\ } \textbf{35} (2015) 585--614.

\bibitem[Tha00]{Thaler00} M.~Thaler, 
{The asymptotics of the {P}erron-{F}robenius operator of a class of
  interval maps preserving infinite measures,} 
 \emph{Studia Math.\ } \textbf{143} (2000) 103--119.
 
\bibitem[Thu05]{Th05}  H.\ Thunberg,
{Positive exponent in families with flat critical point,} 
\emph{Ergod.\ Th.\ and Dyn.\ Sys.\ } \textbf{19} (1999) 767--807.

\bibitem[Z98]{Z98} R.\ Zweim\"uller,
{Ergodic structure and invariant densities of non-markovian interval maps with indifferent fixed points,}  
\emph{Nonlinearity} \textbf{11} (1998) 1263--1276.

\bibitem[Z04]{Z04} R.\ Zweim\"uller,
{S-unimodal Misiurewicz maps with flat critical points,} 
\emph{Fund.\ Math.\ } \textbf{181} (2004) 1--25. 

\bibitem[Z05]{Zwe05} R.\ Zweim\"uller,
{Invariant measures for generalized induced transformation,} 
\emph{Proc.\ AMS} \textbf{138} (2005) 2283--2295. 

\bibitem[Z07]{Zwe07} R.\ Zweim\"uller,
{Mixing limit theorems for ergodic transformations,}
 \emph{J.\ Theoret.\ Probab.\ } \textbf{ 20} (2007) 1059--1071.

\end{thebibliography}
\end{document}